\documentclass[a4paper,12pt,reqno]{amsart}
\usepackage{latexsym, amsfonts, amsmath, amsthm, amssymb, enumerate, mathrsfs, mathtools, tikz-cd, tikz, verbatim, marginnote}
\usepackage[top=3cm, bottom=2.5cm, left=2.5cm, right=2.5cm]{geometry}
\usepackage{hyperref,backref}
\usepackage{longtable}

\renewcommand\Re{\mathop{{\rm Re}}}
\renewcommand\Im{\mathop{{\rm Im}}}
\newcommand{\Db}{{\mathbb D}}

\newcommand{\Cb}{{\mathbb C}}

\newcommand{\Ib}{{\mathbb I}}

\newcommand{\Kb}{{\mathbb K}}

\newcommand{\Nb}{{\mathbb N}}

\newcommand{\Rb}{{\mathbb R}}
\newcommand{\Sb}{{\mathbb S}}

\newcommand{\Xb}{{\mathbb X}}

\newcommand{\Zb}{{\mathbb Z}}
\newcommand{\calA}{\mathcal{A}}
\newcommand{\calB}{\mathcal{B}}
\newcommand{\calC}{\mathcal{C}}
\newcommand{\calD}{\mathcal{D}}
\newcommand{\calE}{\mathcal{E}}

\newcommand{\calG}{\mathcal{G}}
\newcommand{\calJ}{\mathcal{J}}
\newcommand{\calK}{\mathcal{K}}
\newcommand{\calL}{\mathcal{L}}
\newcommand{\calM}{\mathcal{M}}

\newcommand{\calR}{\mathcal{R}}
\newcommand{\calS}{\mathcal{S}}
\newcommand{\calT}{\mathcal{T}}
\newcommand{\calU}{\mathcal{U}}

\newcommand{\scrL}{\mathscr{L}}

\newcommand{\scrS}{\mathscr{S}}

\def\colvec[#1,#2]{\begin{bmatrix} #1 \\ #2 \end{bmatrix}}
\def\rowvec[#1,#2]{\begin{bmatrix} #1 & #2 \end{bmatrix}}
\def\ip<#1,#2>{\left\langle #1,#2 \right\rangle}
\newcommand{\norm}[1]{\left\Vert #1 \right\Vert}

\newcommand{\cl}{\mathop{\mathrm{cl}}}

\newcommand{\Der}{\mathop{\mathrm{Der}}}

\newcommand{\HR}{H\kern-2pt R}

\newcommand{\Rad}{\mathop{\mathrm{Rad}}}

\newcommand{\spann}{\mathop{\mathrm{span}}}

\newcommand{\trace}{\mathop{\mathrm{trace}}}
\newcommand{\Wr}{\mathop{\mathrm{Wr}}}

\newcommand{\st}{\,:\,}

\newtheorem{thm}{Theorem}[section] 
\newtheorem*{thm*}{Theorem}
\newtheorem{cor}[thm]{Corollary}
\newtheorem*{cor*}{Corollary}
\newtheorem{lem}[thm]{Lemma}
\newtheorem{prop}[thm]{Proposition}

\newtheorem*{con*}{Conjecture}
\newtheorem*{prob*}{Problem}

\theoremstyle{definition}
\newtheorem{defn}[thm]{Definition}
\theoremstyle{remark}
\newtheorem{rem}[thm]{Remark}

\newtheorem{ex}[thm]{Example}

\numberwithin{equation}{section}

\begin{document}
\title{Linear systems, spectral curves and determinants}

\author{Gordon Blower}
\email{g.blower@lancaster.ac.uk}
\address{School of Mathematical Sciences, Lancaster
University, Lancaster LA1 4YF, United Kingdom}

\author[Ian Doust]{Ian Doust}
\email{i.doust@unsw.edu.au}
\address{School of Mathematics and Statistics, University of New South Wales, Sydney, NSW 2052, Australia}


\date{23 September 2024}

\begin{abstract} 

 Let $(-A,B,C)$ be a continuous time linear system with state space a separable complex Hilbert space $H$, where $-A$ generates a strongly continuous contraction semigroup $(e^{-tA})_{t\geq 0}$ on $H$, and $\phi (t)=Ce^{-tA}B$ is the impulse response function. Associated to such a system is a Hankel integral operator $\Gamma_\phi$ acting  on $L^2((0, \infty );\Cb )$ and a Schr{\"o}dinger operator whose potential is found via a Fredholm determinant by the Faddeev-Dyson formula.  Fredholm determinants of products of Hankel operators also play an important role in the Tracy and Widom's theory of matrix models and asymptotic eigenvalue distributions of random matrices. 
This paper provide formulas for the Fredholm determinants which arise thus, and determines consequent properties of the associated differential operators. We prove a spectral theorem for self-adjoint linear systems that have scalar input and output: the entries of Kodaira's characteristic matrix are given explicitly with formulas involving the infinitesimal Darboux addition for $(-A,B,C)$.
Under suitable conditions on $(-A,B,C)$ we give an explicit version of Burchnall-Chaundy's theorem, showing that the algebra generated by an associated family of differential operators is isomorphic to an algebra of functions on a particular hyperelliptic curve.  
\end{abstract}

\maketitle

\section{Introduction}\label{S:Intro}

Fredholm determinants and determinants of multiplicative commutators arise in the theory of continuous-time linear systems $(-A,B,C)$ with state space $H$, a complex separable Hilbert space. Associated to each such system is its impulse response (or scattering) function $\phi(t) = C e^{-tA}B$ and a corresponding Hankel operator $\Gamma_\phi$ acting on $L^2((0,\infty);\Cb)$.

We are motivated here by applications from random matrix theory and the spectral theory of Schr\"odinger operators.
A central role is played by the tau function for the linear system. Writing $\phi_{(x)}(t) = \phi(t+2x)$, we define the tau function by the Fredholm determinant $\tau(x) = \det(I + \Gamma_{\phi_{(x)}})$. 
The tau function provides a connection between the potential in Schr\"odinger's equation.
Starting from a Schr\"odinger equation $-f'' + uf = \lambda f$, Faddeev and Dyson showed that under suitable hypotheses, given its scattering function $\phi$ one could recover the potential $u$ via the formula $u(x) = -2 \frac{d^2\ }{dx^2} \log \tau(x)$.

Figure~\ref{relations} expresses the connections between various functions and key equations in the theory. The left-hand column involves Hankel operators built from the scattering data. The scattering function $\phi$ is said to realised by a linear system if there exists $(-A,B,C)$ such that $\phi (x)=Ce^{-xA}B;$ in the terminology of linear systems, this is the impulse response function. Scattering functions can be so realised under hypotheses discussed in \cite{EM}, \cite{McK} and \cite{B1}, as in Example \ref{scatteringex}.

Under mild hypotheses, we show in Theorem \ref{Lyapunovthm} that $R_x=\int_x^\infty e^{-tA}BCe^{-tA} \,dt$ gives a family of trace class operators on $H$ such that $\det (I+R_0)=\det (I+\Gamma_\phi )$. The $R_x$ may be defined as compositions of shifted reachability and controllability operators as in \cite[p.~318]{B1}, although it is more illuminating to observe that they satisfy Lyapunov's equation 
\begin{align}\label{Lyapunov}
    {\frac{dR_x}{dx}}                   &=-AR_x-R_xA\nonumber\\
     \Bigl({\frac{dR_x}{dx}}\Bigr)_{x=0}&=-BC.
\end{align} 
This opens the route towards a family of algebraic identities linking various operators.
The middle column gives the route from the linear system down to the solution of the corresponding Schr\"odinger equation via the Gelfand-Levitan equation (\ref{gelfandlevitanequation}), which we solve using the integral kernel $T_{GL}$. The formula $T_{GL}(x,y)=-Ce^{-xA}(I+R_x)^{-1}e^{-yA}B$ motivates the choice of $R_x$. 

On the right, we introduce the xi function and diagonal Green's function, which are closely related to the spectral data and determine those equations that are integrable. Formal definitions of these functions are given below.

\newcommand{\DGen}{\vbox to 9mm{
  \hbox{diagonal Green's} 
  \hbox{\ \ \ $G(x,x;\lambda)$} }}
\newcommand{\Daren}{\vbox to 9mm{
  \hbox{\ \ \ \ Darboux} 
  \hbox{$(d/dx)G(x,x;\lambda)$} }}

\begin{figure}\label{relations}
\begin{tikzcd}[column sep=huge]
\hbox{scattering $\phi$}  \arrow[d, "\hbox{det}"] 
  & (-A,B,C)  \arrow[l]{} \arrow[r]{}  
     \arrow[d,"\hbox{\small Lyapanov eqn}"] 
  & \hbox{$\xi$ function} \\
\tau   \arrow[d]{} 
  & R_x  \arrow[l,"\hbox{\small Gelfand--Levitan}" below]
      \arrow[r, "\lfloor \cdot \rfloor" below]
      \arrow[d]{}
  & \DGen  \arrow[d]{} \\
  \hbox{potential $u$} \arrow[d]{}
   & T_{GL}(x,y) \arrow[d]{}
   & \Daren \\
   \hbox{Schr{\"o}dinger operator} & \hbox{Baker--Ahkiezer} & 
\end{tikzcd}
\caption{}
\end{figure}

Tracy and Widom \cite{TW} introduced a systematic theory of matrix models involving Fredholm determinants of products of Hankel operators. Unfortunately  the space of Hankel integral operators does not have an obvious multiplicative structure, which inhibits calculations. In previous papers \cite{B1},\cite{B2},\cite{BN} and \cite{BD}, we introduced various tools to reveal the latent algebraic structure of the Hankel operators. In particular we introduced a differential ring of operators $\calE$ on $H$, which in general is not commutative and not self-adjoint, but which does provide a tool for computing the Fredholm determinant of the Hankel integral operators. There is a bracket operator $\lfloor\, \cdot\, \rfloor$ taking $\calE$ to complex functions on $(0, \infty )$.

In section \ref{S:Linearsystems}, we recall some basic results concerning continuous-time linear systems and associated Hankel operators.  
We develop ideas from \cite{BD}, and obtain Fredholm determinants via linear systems that have input and output space $\Cb$. In Proposition \ref{diagonaldiff}, we obtain a sufficient condition for an integral kernel to be the product of Hankel operators.  

In \cite{TW}, many formulas involve products of Hankel integral operators that are associated with first order linear ordinary differential equations with $2\times 2$ matrix coefficients. In section \ref{S:Intkernels}, we show differential equations produce products of Hankel operators and compare some associated functions such as De Branges's phase function with the tau function of a linear system.  
 Products of Hankel operators arise as noncommutative differentials, and so we begin section~\ref{S:Linearsystems} 
 with a discussion of this and the role that the ring of stable rational functions plays in the theory.
Some of our computations for products of Hankel operators involve the diagonal derivative of integral kernels. In section~\ref{S:Extensions} we provide a framework for considering these objects which we use in the subsequent parts of the paper. 

Let $(-A,B,C)$ be a linear system as above and let $\calL^1(H)$ denote the ideal of trace class operators on $H$. A central role in the theory is played by the subalgebra $\calE_0$ of $\calL (H)$ that is generated by $I, BC$ and $(\lambda I-A)(\lambda I+A)^{-1}$ and the corresponding quotient algebra $\calA_0=(\calE_0+\calL^1(H))/\calL^1(H)$. The algebra $\calA_0$ is commutative and unital, and hence may be regarded as a space of functions on its maximal ideal space $\Xb$. For the specific examples we have in mind, it is helpful to regard $\Xb$ as a spectral curve. In particular cases, we can identify $\Xb$ with a curve in the sense of algebraic geometry, and introduce suitable tau functions over the curve as in \cite{Mum} and \cite{Mum2}. These ideas, which lead to a number of determinant formulas, are examined in section~\ref{S:LinearsystemsFredholm}.

In section~\ref{S:HowlandFredholm} we examine a important family of examples of $(-A,B,C)$ which come from Howland operators on $L^2((0,\infty);\Cb)$. Our results explain why the spectral theory of Howland operators is closely linked to that of Hankel integral operators and Schr\"odinger operators. We show that the systems which come from these operators satisfy the hypotheses of the results in section~\ref{S:LinearsystemsFredholm} and hence give specific examples of the determinant formulas obtained.

Associated to $(-A,B,C)$ is family of linear systems
\begin{equation}
   (-A,B,C) \mapsto \Sigma_\zeta =(-A, (\zeta I+A)(\zeta I-A)^{-1}B,C)\qquad (\vert \arg (-\zeta )\vert <\pi /2).
\end{equation} 
which are related to the original $(-A,B,C)$ by the process of Darboux multiplication, as in \cite{BN} and \cite[Section 2]{BD}. 
 As discussed in \cite[Proposition 2.5]{BD} and \cite{McK}, the diagonal Green's function (\ref{diagonalgreen}) describes the infinitesimal Darboux addition.  We then introduce the potential $u(x)=-2{\frac{d^2}{dx^2}}\log\tau (x)$ and the Schr\"odinger differential operator $L=-{\frac{d^2}{dx^2}}+u$.
Taking $\psi_\zeta$ to be a positive solution in $L^2((0,\infty);\Cb)$ of 
$L \psi= -\zeta^2 \psi$,
which is unique up to positive multiples, the effect of replacing $\Sigma_\infty$ by $\Sigma_\zeta$ is replacing 
$u \mapsto u-2{\frac{d^2}{dx^2}}\log \psi_\zeta (x)$.
Ercolani and McKean \cite{EM} considered potentials on a multiplier curve, possibly of an exotic kind, and associated each potential with a divisor. The addition rule for divisors is associated with the Darboux addition of potentials, and the additive class consists of those potentials that arise from Darboux addition as in \cite{McK}. 
In section~\ref{S:xi} we show that the operator $(\lambda I-L)^{-1}$ may be expressed via a Green's function $G(x,y;\lambda )$ (see Equation~(\ref{GreenWronskian})), and the diagonal $G(x,x;\lambda )$ may be expressed in terms of the linear system (Equation~(\ref{diagonalgreen})). 
The infinitesimal version of Darboux addition is given by the diagonal Green's function via $X(u)=-2(d/dx) G(x,x;\lambda )$, as discussed in \cite[Proposition 2.5]{BD} and \cite{McK}.

In the context of Schr\"odinger's equation on the real line, Gesztesy and Simon \cite{GS} introduced the xi function and established the formula $\xi (x;\lambda )=(1/\pi)\Im \log (-G(x,x;\lambda ))$. The xi function is a convenient intermediary between the scattering data and the potential and can be used to solve the inverse spectral problem. These connections are pursued in section~\ref{S:xi}.

In section~\ref{S:Spectral} we give a sufficient condition for $(-A,B,C)$ to produce an integrable Schr\"odinger equation, and
provide an explicit version of Burchnall--Chaundy's theorem, showing how a potential that satisfies the finite stationary KdV hierarchy gives a spectral curve $\Xb$ of finite genus $\ell$. The proof uses identities which originate from Lyapunov's equation (\ref{Lyapunov}).


\section{Linear systems and Hankel operators}\label{S:Linearsystems}

Throughout, let $H$ be a separable complex Hilbert space, regarded as the state space. We write $H'=\calL (H;\Cb )$. Let $\calL (H)$ be the algebra of bounded linear operators on $H$ with the operator norm and the adjoint $T\mapsto T^\dagger$. We shall let $\calL^2(H)$ denote the space of Hilbert-Schmidt operators on $H$, which is a Hilbert space when equipped with the usual inner product $\langle K,L\rangle = \trace(KL^\dagger)$. Let $\calL^1(H)$ denote the space of trace class operators. 

Let $H_0$ be another complex Hilbert space, which serves as the input and output space for our linear system $(-A,B,C)$, so that the input operator is the bounded linear operator $B: H_0\rightarrow H$ and the output operator is the bounded linear operator $C:H\rightarrow H_0$. Let $-A$ be the infinitesimal generator of a strongly continuous $(C_0)$ semigroup $(e^{-tA})_{t\geq 0}$ of linear contractions on $H$. Let $\calD (A)$ be the domain of $A$, which is a dense linear subspace of $H$ and itself a Hilbert space for the graph norm $\Vert f\Vert_{{\calD}(A)}^2=\Vert f\Vert^2_H+\Vert Af\Vert^2_H$. 

For bounded input $u:[0, \infty )\rightarrow H_0$, output $y:[0, \infty )\rightarrow H_0$ and state $x:[0, \infty )\rightarrow H$, the continuous time linear system is governed by the ordinary differential equation
\begin{align}\label{linsystemsODE}
   {\frac{dx}{dt}}&=-Ax+Bu\nonumber\\
                 y&=Cx\nonumber\\
              x(0)&=0.
\end{align}
The impulse response function is $\phi (t)=Ce^{-tA}B$, which gives a weakly continuous function $(0, \infty )\rightarrow \calL (H_0)$, also known as the scattering function. (In many cases we shall take $H_0$ to be $\Cb$ and so $\phi$ is scalar-valued.)

\begin{defn}
Suppose that $\phi\in L^2((0, \infty );\calL(H_0))$.  The Hankel integral operator with impulse response function $\phi$ is
 \begin{equation}\label{Hankelint} 
   \Gamma_\phi f(x) = \int_0^\infty \phi (x+t) f(t)\, dt
                \qquad (f : (0, \infty ) \to H_0 ).
  \end{equation}
\end{defn}

If $\Gamma_\phi$ is trace class on $L^2((0, \infty ); H_0 )$, then the Fredholm determinant $\det (I+\Gamma_\phi )$ is defined, and we shall be concerned with the problem of computing this determinant in cases of interest. 

Given a complex unital algebra $\calR$, we say that $H$ is a (left) Hilbert module if there is a unital homomorphism $\Phi :\calR\rightarrow \calL(H)$ which provides a natural pairing $(a, \eta )\mapsto \Phi (a)\eta$ for all $a\in \calR$, $\eta\in H$; see \cite{DP}. The crucial step is to introduce a differential ring of operators on $H$. 
 Our results on commutative algebras of ordinary differential operators are also related to those of \cite{DGP} and \cite{SW}, except that we use $(-A,B,C)$ as the fundamental datum instead of Grassmannian of subspaces of Hilbert space. Via (\ref{Lyapunov}) we address the functional identities more directly.

Now we consider the stable rational functions, which are fundamental to linear systems theory as in \cite{V} and \cite{vdPS}. Suppose that $(-A,B,C)$ is a linear system with finite-dimensional state space $H$. Then the linear system is effectively described by the transfer function 
  \[ T (s)=D+C(sI+A)^{-1}B, \]
which is a proper rational function, as in \cite{B3}. If $(e^{-tA})_{t\geq 0}$ is of exponential decay, then $T(s)$ is a stable rational function. Let $\scrS$ be the ring of stable rational functions and $\calR$ the subring of $\scrS$ given by  $\calR= \Cb [\lambda ]$ where $\lambda =1/(1+s)$. Both $\calR$ and $\calS$ are principal ideal domains, so every prime ideal is maximal and they have Krull dimension one.  
The quotient field of $\calR$ is equal to $\Cb (s)$. The state space $H$ is then a finitely-generated torsion module over the principal ideal domain $\calR$, and hence has a Jordan decomposition. The Laplace transformation of these variable is simple. Starting from the formula $n!/(1+s)^{n+1}=\int_0^\infty t^{n}e^{-(s+1)t}dt$, we can conveniently express elements of $\calR$ in term of the Laguerre basis of $L^2((0, \infty ); \Cb) $; see \cite{B3} and \cite[(5.36)]{TW}. In Lemma~\ref{boundedHankel}, we obtain conditions for Hankel operators on $L^2((0, \infty ); \Cb )$, defined via the Laplace transform, to be bounded or Hilbert-Schmidt. In Proposition~\ref{diagonaldiff} we obtain a characterization of products of Hankel operators in terms of P\"oppe's semi-additive operators; see \cite{Po}. These are interesting in their own right, and lead to cocycle formulas as in Theorem~\ref{Toeplitzcocycle}.
Conversely, Fuhrmann's theorem \cite[p.~32]{Fu} shows that any proper rational function on $\Cb_\infty$ arises as a transfer function with finite-dimensional state space. 

We now consider the consequences for operators. 

\begin{defn}
For $h \in L^2((0, \infty );\Cb)$ we shall let
$\scrL h(s)= {\hat h}(s) = \int_0^\infty e^{-st}h(t)\, dt$, $\Re s > 0$, denote the Laplace transform of $h$. 
\end{defn}

\begin{defn}\label{Dirichletalgebra} 
\begin{enumerate}[(i)]
\item Let ${\calL}(H)$ be the algebra of bounded linear operators on $H$ with the operator norm and adjoint $T\mapsto T^\dagger$. 
\item  Let  ${\calL}^2(H)$ be the space of Hilbert-Schmidt operators on $H$ with the usual inner product $\langle K,L\rangle = \trace(KL^\dagger)$. 
\item Let $\calS$ be a domain in $\Cb$ with a $C^\infty$ smooth boundary $\partial\calS$, and let $ds$ be the arc\-length measure on $\partial\calS$. Let $\calD_\calS$ be the subalgebra of $L^\infty(\partial \calS;\Cb)$   
such that the norm
  \begin{equation}\label{Dirichlet} 
  \Vert h\Vert^2_{\calD} 
    = 
       \norm{h}_\infty^2
       +\iint_{\partial\calS  \times \partial\calS } 
         {\frac{\vert h(x)-h(y)\vert^2}{\vert x-y\vert^2}} {\frac{ds(x)ds(y)}{\pi}}
  \end{equation}
is finite. Saitoh \cite{S} section 5 discussed equivalent conditions for finiteness of this integral. 
\item  For $h \in L^\infty(\partial\calS ; \Cb))$ let $M_h \in \calL(L^2(\partial\calS ;\Cb )$ denote the operator of multiplication by $h$, namely $f\mapsto hf.$
\end{enumerate}
\end{defn}

\begin{defn} Let $V_1,V_2$ be a vector subspaces of $\calL (H)$. 
\begin{enumerate}[(i)]
\item We write $V_1^\dagger =\{ X^\dagger: X\in V_1\}$ and say that $V_1$ is self-adjoint if $V_1^\dagger= V_1$.
\item With $[a,b]=ab-ba$ for the additive commutator, we write $[V_1,V_2]=\spann \{ [a,b]: a\in V_1, b\in V_2\}$ for the commutator subspace.
\end{enumerate}
\end{defn}

The Dirichlet integral makes precise the heuristic formula $\int_{-\infty}^\infty \vert\kappa\vert e^{-i\kappa x}d\kappa =-2/x^2$ in formulas (\ref{DouglasDirichlet}) and (\ref{DirichletRS}).

Since ${\calL}^2(H)$ is an ideal in ${\calL}(H)$ under composition of operators, We observe that 
\[ {\calA}=\Biggl\{ \begin{bmatrix} U&V\\ X&Y\end{bmatrix}:  U,Y\in {\calL}(H); V,X\in {\calL}^2(H)\Biggr\}\]
is an algebra under the usual multiplication.

\begin{prop} Let $\calS =\{ z\in \Cb: \Re z>0\}$ be the right half plane with boundary $i\Rb$, let $\calD = \calD_\calS$, and let $H = L^2(i\Rb ;\Cb )$. Given $h \in H$, let $\phi_h:(0,\infty) \to \Cb$ be its Fourier transform, $\phi_h (\kappa ) = \int_{-\infty}^\infty h(iy) e^{-i\kappa y} \,dy$.  
\begin{enumerate}[(i)]
\item The map $\calD \rightarrow {\calA}$, given by
   \[ f\mapsto \begin{bmatrix} M_f& D\circ d f\\ 0&M_f\end{bmatrix}\]
is a algebra homomorphism.
\item The map $h \mapsto \Gamma_{\phi_h}$ (as in (\ref{Hankelint}) is a bounded linear map from $\calD$ to  ${\calL}^2(L^2((0, \infty );\Cb) )$. 
\end{enumerate}
\end{prop}

\begin{proof} (i) For $f,h\in \calD$, the differential $f\,dh$ gives rise to an integral operator such that
\[ \iint_{\Rb^2} \vert f(ix)\vert^2 \Bigl\vert {\frac{h(ix)-h(iy)}{x-y}}\Bigr\vert^2 {\frac{dxdy}{\pi}}<\infty .\]

(ii) The Fourier transform $\phi_h (\kappa )$ of $h (iy) \in L^2(i\Rb ; \Cb )$ satisfies
  \begin{equation}\label{DouglasDirichlet} \iint_{\Rb^2} {\frac{\vert h(ix)-h(iy)\vert^2}{\vert x-y\vert^2}} {\frac{dxdy}{\pi}} 
     = \int_{-\infty}^\infty \vert \kappa\vert
                     \vert\phi_h (\kappa )\vert^2 {\frac{d\kappa}{\pi}}.
  \end{equation}
This follows from Plancherel's formula applied to 
 \[ h(iy+it)-h(iy) = {\frac{1}{2\pi}}\int_{-\infty}^\infty
            (e^{i(y+t)\kappa }-e^{iy\kappa}) \phi_h (\kappa )d\kappa . \]
Hence the Hankel integral operator with kernel $\phi_h (\kappa +\eta )$ is Hilbert--Schmidt. Note that for $h$ in the Hardy space $H^2(LHP )$ of the left half plane, $\phi_h (\kappa )=0$ for $\kappa <0$, and 
$\int_0^\infty \xi \vert \phi_h (\kappa )\vert^2 \, d\kappa = \norm{\Gamma_{\phi_h}}^2_{{\calL}^2}$.
\end{proof}

By the Paley--Wiener Theorem, $\scrL: L^2((0, \infty );\Cb) \rightarrow H^2(RHP)$ gives a unitary operator. In (\ref{Hankelint}), $\Gamma_\phi$ is a bounded linear operator $L^2((0, \infty );\Cb )\rightarrow L^\infty ((0, \infty ); \Cb )$. If $t^{1/2} \phi (t) \in L^2((0, \infty );\Cb  )$ then $\Gamma_\phi$ is 
a Hilbert--Schmidt operator on $L^2((0, \infty );\Cb) $. The following result links $\scrL$ with $\Gamma_\phi$.

\begin{lem}\label{boundedHankel} 
 Let $h\in L^\infty (0, \infty) \cap L^2((0,\infty); \Cb )$ and let $\phi = \scrL h$. Then
\begin{enumerate} [(i)] 
  \item $\Gamma_\phi = \scrL M_h \scrL$  and so gives a bounded Hankel integral operator on $L^2((0, \infty );\Cb )$. 
  \item Suppose further that $\int_0^\infty t^{-1}\vert h(t)\vert^2\, dt$ converges. Then $\Gamma_\phi$ is Hilbert--Schmidt.
\end{enumerate}
\end{lem}

\begin{proof} \begin{enumerate} [(i)]
\item Let $P_n(t)$ be the Laguerre polynomial of degree $n$ and let $\ell_n(t)=\sqrt{2} e^{-t}P_n(2t)$. Then  $(\ell_n)_{n=0}^\infty$ forms an orthonormal basis for $L^2((0, \infty );\Cb )$ which gives a Hankel matrix
\begin{align} 
  \bigl\langle\Gamma_\phi \ell_n, \ell_m\bigr\rangle_{L^2((0, \infty );\Cb)}
    &= \int_0^\infty h (t){\frac{2(t-1)^{n+m}}{(t+1)^{n+m+2}}}\, dt\nonumber\\
    &= \int_{-1}^1 h\Bigl( {\frac{1+u}{1-u}}\Bigr) u^{n+m}\, du
          \qquad (n, m=0, 1, \dots ).
\end{align}

For $f \in L^2((0,\infty);\Cb) $,
\begin{align} 
 \scrL M_h \scrL f(x) 
  &= \int_0^\infty e^{-xs}h(s) \int_0^\infty e^{-st}f(t)\, dt ds \\
    &= \int_0^\infty \int_0^\infty e^{-s(x+t)} h(s) \, ds\, f(t) \, dt
                      \nonumber\\
    &=\int_0^\infty \phi (x+t) f(t)\, dt.
\end{align}
Thus $\Gamma_\phi = \scrL M_h\scrL$ is a composition of bounded operators and  hence is bounded.
\item The Laplace transform $\scrL$ on $L^2((0, \infty );\Cb )$ is a self-adjoint operator, with square $\scrL^2$ given by Carleman's Hankel operator \cite[Theorem 2.6]{Power} with integral kernel $1/(x+y)$ on $L^2((0, \infty );\Cb )$ which is known to have spectrum $[0, \pi ]$.  
We have equality of spectra as operators on $L^2((0, \infty );\Cb )$
  \[\sigma (\Gamma_\phi )\setminus \{ 0\}
      = \sigma (\scrL M_h\scrL)\setminus \{0 \}
      =\sigma (M_h\scrL^2) \setminus \{0 \},\]
where $M_h \scrL^2$ has kernel $h(t)/(x+t)$ as an integral operator on $L^2((0, \infty );\Cb )$, and this kernel is Hilbert--Schmidt. We momentarily assume that $\phi$ is real-valued; otherwise, we split into real and imaginary parts. 
Then by \cite[Theorem~8.1]{Si}, the singular numbers satisfy
\[ \sum_{n=0}^\infty \mu_n (\scrL M_h \scrL)^2\leq \sum_{n=0}^\infty \mu_n(M_h\scrL^2)^2 \]
so $\scrL M_h\scrL$ is also Hilbert--Schmidt. \qedhere
\end{enumerate}
\end{proof}

\begin{defn} \label{shift} 
\begin{enumerate}[(i)]
\item For $t > 0$, let $S_t \in \calL(L^2((0, \infty);\Cb )$ denote the shift operator $S_tf(x)=f(x-t){\Ib}_{(0, \infty )}(x-t)$,  $x > 0$, and let $S_t^\dagger$ be its adjoint. 
\item \cite{Po} 
Let $(K_t)_{t>0}$ be a family of integral operators on $L^2((0, \infty );\Cb )$ with kernels $(k_t)_{t>0}$, so $K_tf(x)=\int_0^\infty k_t(x,y)f(y) \,dy$. We say that $(K_t)_{t>0}$ is semi-additive  if there exists a function $k$ such that $k_t(x,y)=k(x+t,y+t)$ for all $x,y,t > 0$. 
\item For a differential function of two variables, we write $\partial_\Delta $ for the diagonal derivative, so $\partial_\Delta k(x,y)= ({\frac{\partial}{\partial x}}+{\frac{\partial}{\partial y}})k(x,y)$.
\end{enumerate}
\end{defn}

The families $(S_t)_{t>0}$ and $(S_t^\dagger )_{t>0}$ both give strongly continuous contraction semigroups on $L^2((0, \infty );\Cb )$, and $(S_t)_{t>0}$ is a semigroup of isometries. Note that for $\phi\in L^2((0, \infty );\Cb )$ the space $Y_\phi =\cl \spann\{ \phi (x+t) :t>0\}$ gives a closed linear subspace of $L^2((0, \infty ); \Cb )$ that is invariant under the adjoint shift semigroup $(S_t^\dagger )_{t>0}$, so $Y_\phi^\perp =L^2\ominus Y_\phi$ is invariant under the shift semigroup $(S_t)_{t>0}$. The space $Y_\phi^\perp$ is the null space of $\Gamma_\phi^\dagger$ since  
  \begin{align*} h\in Y_\phi^\perp
   &\iff \int_0^\infty \phi (x+y)^\dagger h(y) \,dy=0\qquad (x>0)\\
   & \iff \int_{-\infty}^\infty \overline{\hat\phi_{(x)}(iy)}{\hat h(iy)} \, {\frac{dy}{\pi i}}=0\qquad (x>0).
  \end{align*}
  
Hankel integral operators $\Gamma$ on $L^2((0, \infty );\Cb )$ are characterized by the intertwining relation $S_t^\dagger\Gamma=\Gamma S_t$ for all $t>0$. P\"oppe \cite{Po} called such operators additive in view of $\phi (x+t)$ in their kernels (\ref{Hankelint}). The notion of semi-additivity is related to products of Hankel operators, as in Proposition \ref{diagonaldiff} below.

For ease of notation we shall write $\calL^2$ for ${\calL}^2(L^2((0, \infty );\Cb ))$ in the lemma below. 
For $t > 0$ define $\sigma_t: \calL^2 \to \calL^2$ by $\sigma_t(K) = S_t^\dagger K S_t$. A small calculation shows that if $K \in \calL^2$ has kernel
$k(x,y)\in L^2((0, \infty )\times (0,\infty );\Cb )$ then $\sigma_t(K) = S_t^\dagger K S_t$ is the integral operator with kernel $k(x+t,y+t)$. Let $\tilde\sigma_t(K)=S_t K S_t^\dagger$.


\begin{prop}\label{diagonaldiff} 
Let $(\sigma_t)_{t > 0}$ and $(\tilde\sigma_t)_{t>0}$ be as above.
Suppose that $K,L \in \calL^2$. Then
\begin{enumerate}[(i)] 
  \item 
  $(\sigma_t)_{t>0}$ is a strongly continuous contraction semigroup on $\calL^2$.
  \item $\sigma_t(K)\rightarrow 0$ in the strong operator topology as $t\rightarrow\infty $.
  \item the infinitesimal generator of $(\sigma_t)_{t > 0}$ is the diagonal derivative $\partial_\Delta$.
  \item $(\tilde\sigma_t)_{t>0}$ is a strongly continuous contraction semigroup on $\calL^2$.
  \item  $\tilde\sigma_t(KL)=\tilde\sigma_t(K)\tilde\sigma_t(L)$ and $\langle \tilde\sigma_t(K),L\rangle =\langle K, \sigma_t(L)\rangle$.
  \item\label{Hankprod} If $K$ is a product of vector-valued Hankel operators, then the kernel $\partial_\Delta k_t(x,y)$ is of finite rank. 
\end{enumerate}
\end{prop}

\begin{proof} \begin{enumerate}[(i)]
\item This is straightforward. Note that 
   $\Vert S_t^\dagger K S_t \Vert_{\calL^2} 
      \le \Vert {S_t^\dagger} \Vert  \Vert K \Vert_{\calL^2}
                      \Vert {S_t} \Vert = \Vert K \Vert_{\calL^2}$.
\item We have
\begin{equation}\Vert \sigma_t(K)\Vert^2_{\calL^2}=\iint_{(0, \infty )\times (0,\infty )}\vert k(x+t,y+t)\vert^2 \, dxdy\end{equation}
which converges to $0$ as $t\rightarrow\infty$.
\item For suitable differentiable $k$ the operator $d\sigma_t(K)/dt$ has kernel $({\frac{\partial} {\partial x}}+{\frac{\partial}{\partial y}})k(x+t,y+t)$, so $\partial_\Delta$ is the diagonal derivative. 

\item As in (ii), we have 
\begin{equation}\Vert K-\sigma_t(K)\Vert^2_{\calL^2}=\iint_{(0, \infty )\times (0,\infty )}\vert k(x,y)-k(x+t,y+t)\vert^2 \, dxdy\end{equation}
which converges to $0$ as $t\rightarrow 0+$.

\item For the first identity we have that since $S_t^\dagger S_t = I$,
  \[\tilde\sigma_t(KL)= S_tKLS_t^\dagger =S_tKS_t^\dagger S_t L S_t^\dagger =\tilde\sigma_t(K)\tilde\sigma_t(L).\]
For the second, 
\[ \trace(\tilde\sigma_t(K)L^\dagger )=\trace(S_tKS_t^\dagger L^\dagger )= \trace( KS_t^\dagger L^\dagger S_t)=\trace(K\sigma_t(L)^\dagger ).\]

\item Suppose that $\phi, \psi\in L^2((0,\infty );\Cb^n)$ with $\phi$ and $\psi$ absolutely continuous; let $v^\top$ be the transpose of a column vector $v$. Then by integration by parts, we find
\[\partial_\Delta \int_0^\infty \phi (x+z)^\top \psi (z+y)\, dz=-\phi (x)^\top \psi (y). \qedhere\]
\end{enumerate}
\end{proof}



\section{Integral kernels that arise from differential equations}\label{S:Intkernels}

We now contrast and compare operators $K$ with kernel (\ref{k}) and $W$ in (\ref{kerW}) that arise in different ways from a differential equation. Let $J=\begin{bmatrix}0&-1\\ 1&0\end{bmatrix}$ and 
consider the canonical or Hamiltonian differential equation
  \begin{equation}\label{Hamiltonian} 
  {\frac{d}{dx}}J\Psi (x;\lambda ) =
      -(\lambda \Omega_1(x)+\Omega_0(x))\Psi (x;\lambda)
  \end{equation}
where $\Psi (x;\lambda) \in \Cb^{2\times 1}$ and the continuous coefficients in $M_{2\times 2}(\Cb )$ are $\Omega_0(x)=\Omega_0(x)^\dagger$, and a positive semi definite  $\Omega_1(x)=\Omega_1(x)^\dagger$ and $\Omega_1(x)\geq 0$. Let $\Phi (x;\lambda )$ be the fundamental solution matrix. Then
$\Psi (x;\lambda )=\Phi (x;\lambda )[a(\lambda ), b(\lambda )]^\top$ gives a solution with $\Psi (0;\lambda )=[a(\lambda ), b(\lambda )]^\top$. 
 
First consider Hankel factorization, which produces kernels as in Proposition \ref{diagonaldiff}. Suppose $\lambda \geq 0$, but let $x=\Re z$ for $z\in \Cb$, then suppose $\Omega (z;\lambda )=\lambda\Omega_1(z)+\Omega_0 (z)$ satisfies
   \begin{gather*} \Omega (x)=\Omega (x)^\dagger\qquad (x\geq 0) \\
   {\frac{\Omega (z)-\Omega (z)^\dagger }{2i}}\geq 0\qquad (\Im z>0),
   \end{gather*}
in the sense of being positive semidefinite; such an $\Omega$ is matrix monotone and has the form
\begin{equation}\label{omega} 
   \Omega (z)=\Omega_1z+\Omega_0+\int_{\delta}^\infty \Bigl( {\frac{u}{1+u^2}}-{\frac{1}{u+z}}\Bigr) \, \varpi (du)
       \qquad (z\in \Cb \setminus (-\infty , -\delta ])
\end{equation}
where $\Omega_1, \Omega_0$ are constant matrices with $\Omega_1\geq 0$, $\delta>0$ and $\varpi$ is a positive matrix measure on $[\delta, \infty )$. Then by \cite[Theorem 1.1]{B0}, there exists a Hilbert space $H_0$, and $\phi_\lambda\in L^2((0, \infty ); H_0)$ such that
   \begin{equation}\label{k}
   k(x,y; \lambda )
   = {\frac{\Psi (y;\lambda )^\top J\Psi (x;\lambda )}{x-y}} 
   = \int_0^\infty \langle \phi_\lambda (x+u), \phi_\lambda 
                       (y+u)\rangle_{H_0} \, du,
    \end{equation}
which is a product of Hankel operators, giving a particular instance of Proposition \ref{diagonaldiff}. Although the numerator of $k$ vanishes when $x=y$, the right-hand side is jointly continuous as a function of $(x,y)$, so the quotient in (\ref{k}) is unambiguous. 
\begin{ex} By taking $u_j\in (-\infty ,0)$ for $j=-1, \dots ,-N$ and $\Omega_j\in M_{2\times 2}(\Cb )$ with $\Omega_j\geq 0$, we can introduce the rational function
\[ \Omega (x)=\Omega_1x+\Omega_0+ \sum_{j=-1}^{-N} \Bigl( {\frac{u_j}{1+u_j^2}}-{\frac{1}{u_j+x}}\Bigr)\Omega_j\]
which produces examples of the form (\ref{omega}). The Airy kernel \cite{TW} is given by
\[k(x,y)={\frac{{\hbox{Ai}}(x){\hbox{Ai}}'(y)-{\hbox{Ai}}'(x){\hbox{Ai}}(y)}{x-y}}\]
from the differential equation (\ref{Hamiltonian}) with the choice in (\ref{omega}) of
\[\Omega (x)=\begin{bmatrix}x&0\\ 0&-1\end{bmatrix}.\]
\end{ex}

\begin{lem} Let $K:L^2((0, \infty ); \Cb )\rightarrow L^2((0, \infty ) ;\Cb )$ be the integral operator that has kernel $k(x,y)$, so that $K=\Gamma_\phi^\dagger\Gamma_\phi$ where $\Gamma_\phi :L^2((0, \infty ); \Cb) \rightarrow L^2((0, \infty ); H_0)$ is the Hankel operator that has impulse response function $\phi$, and suppose $\int_0^\infty t\Vert \phi (t)\Vert^2_{H_0}dt$ converges. Then
\begin{equation}\label{detK}
  \det (I+\lambda K) = \det \begin{bmatrix} 
              I                   & \lambda \Gamma_\phi \\
             -\Gamma_\phi^\dagger &  I 
             \end{bmatrix}
        \qquad (\lambda\in \Cb ).\end{equation}
\end{lem}
\begin{proof} Here $\Gamma_\phi$ is Hilbert--Schmidt, and $K=\Gamma_\phi^\dagger\Gamma_\phi$ by (\ref{k}), so $K$ is trace class. The determinant identity follows from unitary equivalence.\end{proof}

In Proposition~\ref{GelfandLevitanProp}, we compute such Fredholm determinants by realizing $\phi$ from a linear system and solving a Gelfand--Levitan equation.
The algebraic interpretation of Hankel products appears in section \ref{S:Extensions}.
The differential equation (\ref{Hamiltonian}) also gives a kernel $W$ in (\ref{kerW}) which turns out to be a replicating kernel on a suitable space of holomorphic functions.  

\begin{defn}\label{phase} 
Suppose that $E$ is an entire function such that 
$\vert E(\lambda )\vert >\vert E(\bar\lambda )\vert$ for all $\lambda\in \Cb$ such that $\Im \lambda >0$. We write $E^*(\lambda )=\overline{E(\bar\lambda )}.$
\begin{enumerate}[(i)]
\item A de Branges's phase function for $E$ is a function $\varphi \in C(\Rb ;\Rb )$ such that $E(\kappa )e^{i\varphi (\kappa )}$ is real for all $\kappa \in \Rb$. (Such a function always exists and is unique up to addition of integer multiples of $\pi$; see \cite[page 54]{deB}.)
\item Let $H^2$ be the usual Hardy space on the upper half-plane $\{ z\Cb :\Im z>0\}$ and $R_+:L^2\rightarrow H^2$ the Riesz projection. Likewise, let $R_{-} = I-R_+$ be the complementary projection onto $\overline {H^2}=L^2\ominus H^2$.
\item Let $W$ be the kernel given by 
\begin{equation}\label{kerW1} 
  W(\nu , \lambda ) 
    ={\frac{E(\nu )\overline{E(\lambda )} -E(\bar\lambda )E^*(\nu )}{2\pi i(\bar\lambda -\nu)}}.\end{equation}
\item Let $K(E)$ be the space of entire functions $h$ such that 
$h/E\in L^2(\Rb; \Cb )$, with inner product
\begin{equation}\label{genPW1} \langle f,h\rangle_{K(E)}=\int_{-\infty}^\infty f(t)\overline{h(t)}{\frac{dt}{\vert E(t)\vert^2}} \end{equation}
and such that
\begin{equation}\label{genPW2} \vert h(z)\vert^2 \leq \Vert h\Vert^2_{K(E)}W(z,z)\qquad (z\in \Cb). \end{equation}
\end{enumerate}
\end{defn}

In what follows, we shall consider the primary variable to be the spectral parameter $\lambda$, while the independent variable $x$ in (\ref{Hamiltonian}) is reduced to an auxiliary parameter. For a solution of (\ref{Hamiltonian}), write $\Psi (x; \lambda )=\begin{bmatrix} f(x; \lambda )\\ -h(x; \lambda )\end{bmatrix}$ where 
$f^* (x; \lambda )=f(x; \lambda )$ and $h^*(x;\lambda ) =h (x;\lambda )$, we introduce 
$E(\lambda )=f(x;\lambda )-ih(x;\lambda )$ and introduce the kernels which are
\begin{equation}\label{kerW} 
  W(\nu , \lambda )= {\frac{f(x; \nu )\overline{h(x; \lambda  )}- \overline{f(x; \lambda )} h(x; \nu )}{\pi (\bar\lambda -\nu )}}.\end{equation}
The diagonal $W(\kappa ,\kappa )$ for $\kappa\in \Rb$ is described as follows.

\begin{prop}\label{DeBphaseprop} Suppose that $\Omega (x;\lambda )=\lambda\Omega_1(x)+\Omega_0(x)$ with $\Omega_0, \Omega_1$ real symmetric and $\Omega_1(x)\geq 0$.  Suppose that $\Psi (x;\lambda )$ gives a holomorphic family of solutions of (\ref{Hamiltonian}) for $x>0$ with $\Psi (0, \lambda )$ nonzero and real. Then, with the notation from (\ref{kerW})
\begin{enumerate}[(i)]
\item de Branges's phase function for $E$ satisfies
\begin{equation}\label{DeBphasederivative} 
   \Vert \Psi (x;\kappa )\Vert^2 {\frac{d\varphi}{d\kappa }}
     = \int_0^x \Psi (y; \kappa )^\top 
                  \Omega_1(y)\Psi (y; \kappa )\, dy
       \qquad (\kappa\in \Rb).
\end{equation}
\item The function $\Theta =E^*/E$ is bounded and holomorphic on the upper half plane and $\varphi (\kappa )=(1/2)\Im \log \Theta (\kappa )$.
\item There is a bounded Hankel operator $\Gamma_{\Theta^*}:H^2\rightarrow \overline{H^2}:$ $f\mapsto R_-(\Theta^*f)$ with nullspace $\Theta H^2$.
\item Let $(\kappa_j)_{j=-\infty}^\infty$ be a real sequence and $\alpha\in \Rb $ be such that $\varphi (\kappa_j)=j\pi +\alpha .$ Then
\begin{equation}\label{sampling}\int_{-\infty}^\infty {\frac{\vert F(t)\vert^2}{\vert E(t)\vert^2}}{\frac{dt}{\pi}}\geq \sum_{j=-\infty}^\infty {\frac{\vert F(\kappa_j)\vert^2}{\vert E(\kappa_j )\vert^2\varphi'(\kappa_j)}}\qquad (F\in K(E)).\end{equation} 
\item Suppose that the limits $\varphi (\infty )=\lim_{\kappa\rightarrow\infty} \varphi (\kappa )$ and $\varphi (-\infty )=\lim_{\kappa\rightarrow\infty} \varphi (-\kappa )$ exist, are finite, and $\varphi (\infty )-\varphi (-\infty )\in \pi \Zb.$  Then $\Gamma_{\Theta^*}$ 
 gives a compact Hankel operator such that $\{ f\in H^2: \Vert \Gamma_{\Theta^*} f\Vert =\Vert f\Vert\}$ is a complex vector space of dimension equal to the winding number about $0$ of $\Theta :[-\infty, \infty ]\rightarrow \Cb$.
\end{enumerate}
\end{prop}

\begin{proof} (i) We verify that we have the conditions of Definition \ref{phase}. Simple calculations show that
\begin{equation} 
  \vert E(\lambda )\vert^2-\vert E(\bar\lambda )\vert^2 
  = 4\Im \bigl( h(x;\lambda )f(x;\bar\lambda )\bigr)
  = {\frac{2}{i}} \Psi (x; \lambda )^\dagger J \Psi (x; \lambda ).
\end{equation}
By hypothesis,  $\Psi (0; \lambda )$ is real. Then $\Psi (0;\lambda )^\dagger J\Psi (0; \lambda )=0$, so for $\lambda =\kappa +i\varepsilon =\bar \nu$, we have 
\begin{align}\label{speckerintegral}
   {\frac{\Psi (x;\lambda )^\dagger J\Psi (x; \lambda )}{2i\varepsilon}}
    &=\int_0^x \Bigl({\frac{d}{dy}}
       {\frac{\Psi (y;\lambda )^\dagger J\Psi (y; \lambda )}{2i\varepsilon}}\Bigr)\, dy\\
   &= \int_0^x \Psi (y; \lambda )^\dagger \Bigl( {\frac{\Omega 
        (x;\lambda )-\Omega (x; \lambda )^\dagger }{2\varepsilon i}}\Bigr)\Psi (y; \lambda )\, dy\\
   &=\int_0^x \Psi (y; \lambda )^\dagger \Omega_1 (y) 
            \Psi (y; \lambda )\, dy, 
\end{align}
We deduce that $\vert E(\lambda )\vert >\vert E(\bar \lambda )\vert$ for all $\Im \lambda >0$, thereby fulfilling the conditions for the relevant function spaces of exist. Also, the diagonal of the kernel is given by 
\begin{align}\label{diagW} W(\kappa ,\kappa )&=\lim_{\varepsilon\rightarrow 0+} 
{\frac{\vert E(\kappa +i\varepsilon )\vert^2-\vert E(\kappa -i\varepsilon )\vert^2}{4\pi \varepsilon}}\nonumber\\
&={\frac{1}{\pi }}\int_0^x \Psi (y; \kappa )^\dagger 
\Omega_1 (y)\Psi (y;\kappa )\, dy\qquad (\kappa\in \Rb ).\end{align}
We observe that for real initial conditions and $\kappa\in \Rb$, the entries of $\Psi (x;\kappa )$ are real, so
\[ \vert E(\kappa )\vert^2=\vert f(x;\kappa )-ih(x;\kappa )\vert^2=\Vert \Psi (x;\kappa )\Vert^2.\]
Now \cite[Problem 48]{deB} is easily solved by applying the Cauchy-Riemann equations, so we obtain
\begin{equation}\label{prob48} {\frac{d\varphi}{d\kappa}}={\frac{\pi W(\kappa,\kappa )}{\vert E(\kappa )\vert^2}},\end{equation}
which combined with (\ref{diagW}) yields the result.\par
\indent 
(ii) We follow \cite[Theorem 4.1]{B0}, and observe that $\Theta=E^*/E$ determines a bounded holomorphic function on $\{ z: \Im z>0\}$ with $\vert \Theta (\kappa )\vert =1$ for almost all $\kappa \in \Rb$. Then $\Theta (\kappa )=\overline{E(\kappa )}/E(\kappa )=e^{2i\varphi (\kappa )}$.\par 
\indent (iii) By Proposition \ref{DeBphaseprop}(ii), $\Theta^*$ gives the conjugate of the function $\Theta\in H^\infty$ on $\Rb$, so we have $ W(z,w)=E(z)\overline{E(w)}K_\Theta (z,w)$ where
\[K_\Theta (z,w)={\frac{1-\Theta (z)\overline{\Theta (w)}}{2\pi i(\bar w-z)}}\]
gives the reproducing kernel for the model space $H^2\ominus \Theta H^2$. 
Then $\Gamma_{\Theta^*}:H^2\rightarrow \overline{H^2}$ $\Gamma_{\Theta^*}f=R_-(\Theta^*f)$ gives a bounded Hankel operator. Now $H^2\ominus \Theta H^2$ is the range of the orthogonal projection $f\mapsto \Theta R_-(\Theta^* f)$, and $\Theta H^2$ is the null space of $\Gamma_{\Theta^*}$. The operation of multiplication by $\Theta^*$ on $L^2(\Rb;\Cb )$ gives a unitary, with block form
\begin{equation} \label{Toeplitzblock}\begin{bmatrix} T_{\Theta^*}& 0\\ \Gamma_{\Theta^*}& \tilde T_{\Theta^*}\end{bmatrix}\qquad \begin{matrix} H^2\\ \overline{H^2}\end{matrix}\end{equation}
where $T_{\Theta^*}\in\calL (H^2)$ is the Toeplitz operator with symbol $\Theta^*$, and $\tilde T_{\Theta^*}\in \calL (\overline{H^2})$ is multiplication by $\Theta^*$.\par 
(iv) As in \cite[Problem 49]{deB}, the functions
\[{\frac{W( z,\kappa_j)}{\overline{ E(\kappa_j)}}} ={\frac{E^*(z)e^{-2i\varphi (\kappa_j)} -E(z)}{2\pi i(z-\kappa_j )}}\]
belong to $K(E)$, and
\[\Bigl\langle {\frac{W( z,\kappa_j)}{\overline{ E(\kappa_j)}}},{\frac{W( z,\kappa_\ell )}{\overline{ E(\kappa_\ell)}}}\Bigr\rangle  ={\frac{E(\kappa_\ell) ( e^{2i\kappa_\ell-2i\kappa_j }-1)}{2\pi i(\kappa_\ell-\kappa_j ) \overline{E(\kappa_j)}}},\]
so the sequence is orthogonal by the choice of sampling sequence. The normalizing constants are given by (\ref{prob48}). The hypotheses allow for $(\kappa_j)$ to be an infinite sequence, so the sampling sequence is infinite. The case of a finite sequence is addressed in case (v), as follows.\par
\indent (v) By the assumptions on $\varphi$, we have $\lim_{x\rightarrow\infty}\Theta^*(x)=\lim_{x\rightarrow\infty}\Theta^*(-x)$, 
so we can extend $\Theta^*$ to a continuous and unimodular function $\Theta^*:[-\infty ,\infty ]\rightarrow \Cb$ such that $\Theta^*(-\infty )=\Theta^* (\infty )$. Hence by Hartman's theorem \cite[Theorem 5.5]{Pe}, $\Gamma_{\Theta^*}$ is compact, hence $T_{\Theta^*}$ is Fredholm; then
\begin{align} \dim{\hbox{null}} (I-\Gamma_{\Theta^*}^\dagger\Gamma_{\Theta^*})& =\dim{\hbox{null}}(T_{\Theta^*}^\dagger T_{\Theta^*})-\dim{\hbox{null}}(T_{\Theta^*}T_{\Theta^*}^\dagger )\nonumber\\
&=\dim{\hbox{null}}(T_{\Theta^*})-\dim{\hbox{null}}(T_{\Theta^*}^\dagger )\nonumber\\
&={\hbox{ind}} (T_{\Theta^*})
\end{align}
which by \cite{Pe} is given by the negative winding number about $0$ of $\Theta^*: [-\infty , \infty ]\rightarrow \Cb$, or equivalently the winding number about $0$ of $\Theta :[-\infty ,\infty ]\rightarrow \Cb$. Also, $\Vert\Gamma_{\Theta^*}\Vert \leq \Vert\Theta\Vert_{\infty}=1$, so 
$R=I-\Gamma_{\Theta^*}^\dagger\Gamma_{\Theta^*}$ satisfies $0\leq R\leq I,$ hence the Cauchy--Schwarz inequality $\vert\langle Rf,g\rangle\vert^2\leq \langle Rf,f\rangle \langle Rg,g\rangle$. Applying this repeatedly, we find that 
\[\{ f\in H^2: \Vert \Gamma_{\Theta^*}f\Vert=\Vert f\Vert\} =\{ f\in H^2: (I-\Gamma_{\Theta^*}^\dagger\Gamma_{\Theta^*})f=0\}\]
gives a vector space of dimension equal to the winding number about $0$ of $\Theta$, namely
  \[\int_{-\infty}^\infty {\frac{\Theta'(\kappa )}{\Theta (\kappa )}}{\frac{d\kappa}{2\pi i}}
=\int_{-\infty}^\infty \varphi (\kappa ){\frac{d\kappa}{\pi}}.
     \qedhere\]
\end{proof}

\begin{ex}
(i) For example, $E(z)=z+i$ gives $\Theta (z)=(z-i)/(z+i)$, so the image of $\Theta :[-\infty , \infty ]\rightarrow \Cb$ gives the unit circle.

(i) With $a>0$ the function $E(z)=\cos az-ia\sin az$ satisfies the conditions of Definition \ref{phase}, and the equations (\ref{genPW1}) and (\ref{genPW2}) determine the Paley-Wiener space of entire functions $PW(a)$. 
In this case, $H^2\ominus \Theta H^2$ has infinite dimension since
\[ \Theta (z)= {\frac{ (1+a) e^{2iaz}+(1-a)}{(1-a)e^{2iaz}+(1+a)}}\]
has infinite winding number about $0$ as $z$ describes $(-\infty ,\infty )$. 
\end{ex}

\begin{ex}\label{Schrodingerasmatrix} For the system
\begin{equation} {\frac{d}{dx}}\begin{bmatrix} 0&-1\\ 1&0\end{bmatrix}\begin{bmatrix}f\\ -h\end{bmatrix} =
\begin{bmatrix}\lambda -u&0\\ 0&1\end{bmatrix}
\begin{bmatrix} f\\ -h\end{bmatrix}\end{equation}
we have ${\frac{df}{dx}}=-h$ and $-{\frac{d^2f}{dx^2}}+uf=\lambda f$, so $E=f+i{\frac{df}{dx}}$. With $f(x; \kappa )$ real for 
$x,\kappa\in \Rb$, the expression $E=e^{-i\varphi }\vert E\vert$ gives the polar form of the curve $\kappa \mapsto f(x;\kappa )+if'(x;\kappa )$, in terms of the spectral parameter. We obtain an expression for this $f$ in Remark~\ref{BakerAkhiezer} via the Gelfand-Levitan equation. In Section \ref{S:xi}, we consider the corresponding Green's function $G(x,y;\lambda )$ and introduce the $\xi$ function (\ref{Krein}) 
from the polar form of $-G(x,x;\lambda )$, which is different from $W(\lambda ,\lambda )$.
\end{ex}

The spectral theory of $K$ is not to be confused with the spectral theory of (\ref{Hamiltonian}). Nevertheless, the differential equation leads to a related kernel via the following standard result.

\begin{prop}\label{specRKHS} Suppose that $\lambda\mapsto \Psi_\lambda $ gives a homomorphic function $\Cb\setminus \Rb \rightarrow L^2((0, \infty ), \Cb^{2\times 1})$ such that $\overline{\Psi_\lambda (x)}=\Psi_{\bar\lambda} (x)$ and 
$\Psi (0;\lambda )=[a(\lambda ), b(\lambda )]^\top.$
Then there exists a reproducing kernel Hilbert space of holomorphic functions on $\Cb\setminus \Rb$ with kernel
\begin{equation}\label{Lambda} \Lambda (\lambda ,\nu )={\frac{\bar a(\nu )b(\lambda )-a(\lambda )\bar b(\nu )}{\lambda-\bar\nu }}.\end{equation}
\end{prop}

\begin{proof} By integrating the differential equation (\ref{Hamiltonian}), we have
  \begin{align}
  (\bar \nu-\lambda )\int_0^x \Psi (y;\nu )^\dagger \Omega_1(y)\Psi (y;\lambda ) \, dy
    &=\int_0^x{\frac{d}{dy}}\Bigl( \Psi (y;\nu )^\dagger J\Psi (y;\lambda ) \Bigr)\, dy\\
    &=\Psi (x;\nu )^\dagger J\Psi (x;\lambda )-  \Psi (0;\nu )^\dagger J\Psi (0;\lambda ).
  \end{align}
Then taking $x\rightarrow\infty$, one has 
  \begin{equation}\label{Lambdaphi}
   \Lambda (\lambda ,\nu )
   = {\frac{\bar a(\nu )b(\lambda )-a(\lambda )\bar b(\nu )}{\lambda-\bar\nu }}
   =\int_0^\infty \Psi (x;\nu )^\dagger \Omega_1(x)\Psi (x;\lambda ) \,dx.
  \end{equation}
We observe that $\Lambda (\lambda ,\nu )=\overline{\Lambda (\nu , \lambda )}$ and $\lambda \mapsto \Lambda (\lambda , \nu )$ is holomorphic. Now by (\ref{Lambdaphi}),  $[\Lambda (\lambda_j, \lambda_k )]_{j,k=1}^N$ is positive semidefinite for all $\lambda_1, \dots ,\lambda_N$ such that $\Im \lambda_j>0$. Hence $\Lambda $ determines a reproducing kernel Hilbert space on $\{ \lambda \in \Cb :\Im \lambda >0\}$, and the elements are holomorphic functions.
\end{proof}

\begin{rem} (i) Note that $f(x,\lambda )=\cos (\sqrt{\lambda }x)$ and $h(x;\lambda )=(\sin (\sqrt{\lambda }x))/\sqrt{\lambda}$ defined via power series are entire, and 
\[ i\sqrt{\lambda}={\frac{-1}{\sqrt{2}}}+{\frac{1}{\pi}}\int_0^\infty \Bigl({\frac{1}{t-\lambda }}-{\frac{t}{1+t^2}}\Bigr)\sqrt{t}\, dt\qquad (\lambda\in \Cb\setminus [0, \infty ))\]
gives a Herglotz function with $\Re i\sqrt{\lambda }<0$ for $\Im \lambda>0$, such that $f(x;\lambda )+i\sqrt {\lambda} h(x;\lambda )$ belongs to $L^2((0, \infty ); \Cb )$ for $\Im \lambda >0$. \par

(ii) For $\Omega_0(x)=0$ for all $x\geq 0$,  the condition
\[\int_0^\infty \trace\,\Omega_1(x)\, dx=\infty\] 
ensures that we have the limit point case, so the system is self-adjoint, for some suitably boundary condition at $x=0$. Then there exists a Herglotz function $m_+(\lambda )$ such that
\begin{equation}\label{Weylm} \Psi (x;\lambda )=\Phi (x;\lambda ) \begin{bmatrix} 1\\ m_+(\lambda )\end{bmatrix} \in L^2( (0, \infty );\Omega_1(x)dx; \Cb^{2\times 1})\end{equation}
as in \cite{LW}; this $m_+(\lambda )$ is known as Weyl's function. The spectral measure $\mu$ is determined by
  \begin{equation}\label{spectralmeasure} 
  m_+(\lambda) = \alpha \lambda +\beta 
    + \int_{-\infty}^\infty \Bigl( {\frac{1}{\kappa -\lambda}}-{\frac{\kappa }{1+\kappa^2}}\Bigr) \, \mu (d\kappa )
   \end{equation}
and we take $m_+(\bar\lambda )=\overline{m_+(\lambda )}.$ The kernel from Proposition~\ref{specRKHS} is given by  
\[ \Lambda (\lambda, \nu )={\frac{m_+(\lambda )-\overline{m_+(\nu )}} {\lambda-\bar\nu }}=
\alpha +\int_{-\infty}^\infty{\frac{ \mu (d\kappa )}{(\kappa -\lambda )(\kappa -\bar\nu )}}.\]
For differential equations on $\Rb$, we introduce a Herglotz function $m_-$ such that 
\begin{equation}\label{Weylm-} \Psi (x;\lambda )=\Phi (x;\lambda ) \begin{bmatrix} 1\\ -m_-(\lambda )\end{bmatrix} \in L^2( (-\infty ,0 );\Omega_1(x)dx; \Cb^{2\times 1});\end{equation}
this convention is common and matches \cite{GS}, but is not universal. 

(iv) 
If the spectrum is absolutely continuous, then the density of states is determined by the diagonal of $\Lambda $ via
\begin{equation}\label{specdensity}2\pi i{\frac{d\mu}{d\kappa }}(\kappa )=\lim_{\varepsilon\rightarrow 0+}\bigl( m(\kappa +i\varepsilon )-
m(\kappa -i\varepsilon )\bigr).\end{equation}
at the Lebesgue points of the spectrum. As discussed in \cite[Section 5]{GS}, the absolute continuity of the spectrum does not change if one moves the boundary condition from $x=0$ to some other point.
\end{rem}

\section{Extensions, cocycles and Hankel products}\label{S:Extensions}

 Products of Hankel operators arise as noncommutative differentials, which are defined as follows.

\begin{defn} \label{differentials} Let $\calR$ and $\calL$ be unital complex algebras.
\begin{enumerate}[(i)]
\item We define the space of noncommutative differentials $\Omega^1 \calR$ as the subspace of $\calR \otimes \calR$ that makes the sequence
  \begin{equation}\label{Omega1again}
   0 \longrightarrow \Omega^1 \calR \longrightarrow \calR \otimes \calR \xlongrightarrow{\mu} \calR \longrightarrow 0
  \end{equation}
exact, where $\mu (a\otimes b)=ab$ is the multiplication map, and $\Omega^1\calR$ is spanned by 
$a\,db = ab\otimes 1 - a \otimes b$ for $a,b\in \calR$. Evidently $\Omega^1\calR$ is a left $\calR$-module. 
\item Let $\rho: {\calR}\rightarrow {\calL}$ be a $\Cb$-linear map, and define its curvature to be $\varpi: {\calR}\otimes {\calR}\rightarrow {\calL}$ where $\varpi (f_1,f_2)=\rho (f_1f_2)-\rho (f_1)\rho (f_2)$, so $f_1df_2\mapsto \varpi (f_1,f_2)$ is a linear map $\Omega^1\calR \rightarrow\calL$. (The curvature is related to the notion of semicommutator on operator theory. This is not to be confused with the curvature of modules discussed in \cite[p.~116]{DP}). The dual space of $\Omega^1{\calR}$ is identified in Definition~\ref{Hochschildcob}. 
  \item
Suppose further that $\calR$ is commutative. Then $\Omega^1R$ may be regarded as an ideal in $\calR\otimes \calR$ with square ideal $(\Omega^1\calR)^2$, so the quotient of ideals $\Omega_\calR^1 = \Omega^1 \calR/(\Omega^1 \calR)^2$
is a left $\calR$-module called the space of first-order K{\"a}hler differentials; see \cite[p.~173]{H}. 
\item The space of $n^{th}$-order K{\"a}hler differentials $\Omega^n_\calR$ is the skew-symmetric subspace $\Omega^1_\calR\wedge \dots\wedge \Omega^1_\calR$ of $\otimes_{j=1}^n \Omega^1_\calR$ with $n$ factors, in which our tensors are formed over the field $\Cb$.  
\end{enumerate}
\end{defn}

\begin{defn} Let ${\calA}$ be an associative unital complex algebra such that $\Omega^1 {\calA}$ is a projective ${\calA}$-bimodule in the sense that $\Omega^1 {\calA}$ is a direct summand of a free ${\calA}$-bimodule. Then ${\calA}$ is said to be \textbf{quasi-free}.
\end{defn}

In Propositions 3.3 and 5.3 of \cite{CQ}, there are various equivalent characterizations of this property.
If $\calA$ is quasi-free, then $M_{n\times n}(\calA )$ is also quasi-free. This enables one to generate examples relevant to linear systems. For ${\calA} = \Cb[x]$, the bimodule $\Omega^1 {\calA}$ is a direct summand of the free ${\calA}$-bimodule ${\calA} \otimes {\calA}$, so $\Omega^1 {\calA}$ is projective and $M_{n\times n}(\Cb [x])$ is quasi-free. 

In Propositions \ref{diagonaldiff} and \ref{DeBphaseprop}, the diagonals of kernels emerge in the course of calculations involving products of Hankel operators. In the following section, we express these in a unified framework which will enable us to formulate results in later sections more easily, such as Theorem~\ref{Lyapunovthm}. 

Suppose
that ${\calA}$ is a commutative and unital Banach algebra, so the projective tensor product ${\calA}\hat\otimes {\calA}$ is likewise. Let $\calA'$ denote the dual space of $\calA$, let $\Xb$ be the maximal ideal space, identified with a subspace of $\calA'$, and let $a\mapsto \hat a$ be the Gelfand transform ${\calA} \rightarrow C({\Xb}; \Cb )$. From the formula $\psi (a_0\otimes b)=\psi (a_0\otimes 1)\psi (1\otimes a_1)$, Tomiyama observed that the maximal ideal space of ${\calA}\hat\otimes {\calA}$ is ${\Xb}\times {\Xb}$. Let ${\Xb}_\Delta =\{ (\phi , \phi ): \phi \in {\Xb}\}$  be its diagonal. 

\begin{prop}\label{Tomi} 
Let ${\calA}$ be semisimple.
\begin{enumerate}[(i)]
\item
 Under the Gelfand transform ${\calA}\hat \otimes {\calA}\rightarrow C({\Xb}\times {\Xb}; \Cb),$  the space $\Omega^1{\calA}$ is mapped into $\{ \hat a\in C({\Xb}\times {\Xb}; \Cb): \hat a\vert {\Xb}_\Delta =0\}$.
\item For every nonzero derivation $\partial:{\calA}\rightarrow {\calA}$, there exists a unique map $a_0 da_1\mapsto \widehat{(a_0\partial a_1)}$ which is nonzero.
\end{enumerate}
\end{prop}

\begin{proof} (i) Here ${\Xb}_\Delta$ is the image of the canonical inclusion ${\Xb}\rightarrow {\Xb}\times {\Xb}:$ $\phi\mapsto (\phi ,\phi)$, which is compatible with $\phi \mapsto\phi\circ \mu$, and $\Omega^1{\calA}$ is the nullspace of $\mu$.

(ii) We can regard ${\calA}$ as a natural ${\calA}$-bimodule for multiplication. Then the map $\Omega^1{\calA}\mapsto {\calA}:   a_0 da_1\mapsto a_0 \partial a_1$ exists by the universal property of $\Omega^1{\calA}$; see \cite[Proposition 8.4.4]{QB}. The Gelfand transform is injective on the semisimple Banach algebra ${\calA}$.  
\end{proof}

\begin{rem} (i) 
The algebra $M_{n\times n}(\calA)$ has centre ${\calA}$ and, as Schelter observed, $M_{n\times n}(\calA)$ is quasi-free if and only if $\Omega^1M_{n\times n}({\calA})$ is locally projective at at each point of ${\Xb}_\Delta$; see \cite[Lemma 3.2]{Sch}. \par
(ii) 
Let ${\calA}$ be a unital and finitely generated complex algebra, so there is an exact sequence of complex algebras and homomorphisms
  \[ 0\longrightarrow J
      \longrightarrow \Cb[x_1, \dots, x_n]
      \longrightarrow {\calA}
      \longrightarrow 0.  \]
Suppose further that ${\calA}$ has no nonzero nilpotent elements, or equivalently $J$ is a radical ideal; then ${\calA}=\Cb[\Xb ]$ is the coordinate ring for some complex variety $\Xb =\{ (y_1, \dots, y_n): f(y_1, \dots, y_n)=0,\, \forall f\in J\}$; see \cite[p.~233]{Sha}. 
It follows from the Hochschild--Konstant--Rosenberg theorem that if ${\calA}$ is quasi-free, then the components of $\Xb $ have dimension $\leq 1$, hence are points and nonsingular affine curves; see  \cite[p.~273]{CQ}. Suppose further that $\Xb$ is irreducible; then ${\calA}$ is an integral domain and its field of fractions $\Cb(\Xb )$ defines the field of rational functions on  $\Xb$; see \cite{Sha}. Thus quasi-free algebras are an appropriate noncommutative analogue of the coordinate algebra on a Riemann surface. See Theorem \ref{BurchnallChaundy}.  
\end{rem}

\begin{defn}\label{extension} Let $\calB$ be any complex Banach algebra. By an extension of $\calB$ we mean a sequence of Banach algebras and continuous complex algebra homomorphisms such that  
    \[ 0\xlongrightarrow{} \calJ
        \xlongrightarrow{} \calE
        \xlongrightarrow{\pi} \calB
        \xlongrightarrow{} 0.\]
is exact. A section of $\pi$ is a continuous linear map $\rho:\calB\rightarrow \calE$ such that $\pi\circ \rho (a)=a$ for all $a\in\calB$.
\end{defn}

\begin{ex}
Let $X$ be any compact metric space, $\calK$ the space of compact operators on $H$ and 
 \[ 
  0 \xlongrightarrow{} \calK
    \xlongrightarrow{} \calE
    \xlongrightarrow{\pi} C(X; \Cb) 
    \xlongrightarrow{0} 0 \]
an exact sequence of $C^*$ algebras and $*$-homomorphisms, giving an extension of $C(X; \Cb )$. Then there exists a positive linear map $\rho :C(X; \Cb )\rightarrow \calE$ such that $\pi \circ \rho (a)=a$ for all $f\in C(X; \Cb )$ by Andersen's theorem \cite{And}. This $\rho$ is generally not a homomorphism. We are interested in extensions such that one can choose $\rho$ to respect a fine structure associated with (non-closed) ideals in $\calK$ such as the Schatten classes, and (non-closed) subalgebras of $C(X; \Cb )$. There is a question of finding subalgebras such that $\rho$ restricts to a homomorphism on the subalgebra, and of measuring how much $\rho$ itself deviates from being a homomorphism. To make these ideas precise, we recall some definitions from cyclic theory.  
\end{ex}

For a complex unital algebra $\calR$, there is a unique differential graded algebra $\Omega \calR=\oplus_{n=-\infty}^\infty \Omega^n \calR$ such that $\Omega^0\calR=\calR$ with a graded differential $d:\Omega^n\calR\rightarrow \Omega^{n+1}\calR$ with the following universal property. Given any differential graded algebra $S$ and homomorphism $u:\calR\rightarrow S$, there is a unique homomorphism of differential graded algebras $\Omega \calR\rightarrow S$ that extends $u$. 

\begin{defn}\label{Hochschildhom} 
Let $b:\Omega \calR\rightarrow \Omega \calR$ be the linear map determined by $b:\Omega^{n+1}\calR\rightarrow \Omega^n\calR$ $b(\omega da)=(-1)^n(\omega a-a\omega )$. Then $b^2=0$, so there is a complex 
  \[0\longleftarrow \calR\longleftarrow \Omega^1 \calR\longleftarrow\Omega^2\calR\longleftarrow \cdots \]
with differential $b$ of order $(-1)$. Then the Hochschild homology $HH_*(\calR)$ is the homology of this complex.
\end{defn} 
 
\begin{rem}\label{HHcalculation} (i) Consider the commutator subspace $[\calR,\calR]=\spann\{ [r,s]: r,s\in \calR\}.$ Then $[\calR,\calR]=\{0\}$ if and only if $\calR$ is commutative, and generally $HH_0(\calR)=\calR/[\calR,\calR]$. We consider this quotient in the next section. By \cite[p.~21]{L}, the matrix algebra $M_{n\times n}(\calR )$ satisfies $HH_j(M_{n\times n}(\calR ))=HH_j(\calR )$ for $j,n=0, 1, \dots$.

(ii) For commutative unital $\calR$, there is an injective map $\Omega^j_\calR\rightarrow HH_j(\calR)$ from the K\"ahler differentials to the Hochschild homology of $\calR$ by \cite[Prop 7.2.3]{QB} for $j=1, 2, \dots $which is 
a bijection for $j=1$ by \cite[Prop 1.1.10]{L}. Nevertheless, the formula for the $1$-cyclic cocycle involves (\ref{Jacobianidentity}), which involves $\Omega^2_\calR$ for $\calR=\calD$.
\end{rem}

We consider the dual space of $\Omega^1{\calR}$ in Definition \ref{Hochschildcob} as in \cite[section 7.2]{QB}.

\begin{defn}\label{Hochschildcob} Let $\calD$ a complex unital algebra.
 \begin{enumerate}[(i)] \item  Let $\varphi_n:{\calD}^{n+1}\rightarrow \Cb$ a multilinear functional. Then the Hochschild coboundary is $(b\varphi_n ):{\calD}^{n+2}\rightarrow\Cb $ given by 
\[ (b\varphi_n )(a_0, \dots, a_{n+1})=\sum_{j=0}^n (-i)^j\varphi_n(a_0, \dots, a_ja_{j+1}, \dots ,a_n)+(-1)^{n+1} \varphi_n(a_{n+1}a_0, \dots, a_n).\]
If $(b\varphi_n )=0$, then we call $\varphi_n$ a Hochschild $n$-cocycle.
\item  Suppose further that $\varphi_n(a_1, \dots a_n,a_0)=(-1)^n\varphi_n (a_0, \dots, a_n)$. Then $\varphi_n$ is a cyclic $n$-cocycle. By \cite[p.~186]{Co}, such a $\varphi_n$ is equivalent to a trace $\tilde\varphi_n:\Omega^n\calD\rightarrow \Cb$ via $\varphi_n(a_0, \dots, a_n)=\tilde\varphi_n (a_0da_1\dots da_n)$.
  \item 
An even normalized cochain is a $\Cb$-linear map $\varphi_{2m}: {\calD}^{2m+1}\rightarrow {\Cb}$ 
such that 
  \[ \varphi_{2m} (a_0,a_1,\dots , a_{2m})=0 \] 
if $a_j=1$ for some $j=1, \dots, 2m$.
\end{enumerate}
\end{defn}

\begin{ex}\label{Manifoldcocylce} Let $\Xb$ be a two dimensional Riemannian manifold with volume measure $d\mu$ and $P:\Xb\rightarrow \Cb$ a function integrable with respect to $\mu$. Then on the algebra $C^\infty (\Xb ; \Cb )$, there is a cyclic $1$-cocycle $\varphi :C^\infty (\Xb ; \Cb)\times C^\infty (\Xb , \Cb )\rightarrow \Cb $ given by
\begin{equation}\label{onecocyclce}
  \varphi (f,h)=\iint_\Xb P(x,y) {\frac{\partial (f,h)}{\partial (x,y)}}\mu \,(dxdy).
\end{equation}
Clearly $\varphi (f,h)=-\varphi (h,f)$ and by elementary calculus,
\begin{equation}\label{Jacobianidentity} 
  {\frac{\partial (f_0f_1,f_2)}{\partial (x,y)}} 
        - {\frac{\partial (f_0,f_1f_2)}{\partial (x,y)}}
        + {\frac{\partial (f_2f_0, f_1)}{\partial (x,y)}}
        =0,\end{equation}
hence $\varphi (f_0f_1,f_2)-\varphi (f_0,f_1f_2)+\varphi (f_2f_0,f_1)=0$, so $b\varphi =0$. 
One can regard (\ref{onecocyclce}) as determining $\tilde\varphi : \Omega^1C^\infty (\Xb ;\Cb )\rightarrow \Cb $,  $\tilde\varphi (fdh)=\varphi (f,h)$, or as a linear functional $\Omega^2_{C^\infty (\Xb; \Cb )}\rightarrow \Cb $ such that $df\wedge dh\mapsto \varphi (f,h)$. We encounter variants of this formula in Proposition \ref{cryptotrace} and Corollary \ref{hyperprincipal}. 
\end{ex}

\begin{defn} Let $\calJ$ be an ideal in $\calL (H)$. Let $\rho :\calD\rightarrow \calL (H)$ a $\Cb$-linear map such that the curvature takes values in $\calJ$. Then $\rho$ is said to be a homomorphism modulo $\calJ$.
\end{defn}

\begin{ex} Let $\calJ={\calL}^m(H)$ so the ideal generated by $m$-fold products of elements of $\calJ$ is $\calJ^m={\calL}^1(H)$, the trace class operators and suppose that $\varpi$ takes values in $\calJ$, so
\[ \varphi_{2m} (a_0, a_1, \dots ,a_{2m})=\trace\bigl( \rho (a_0)\varpi (a_1, a_2) \dots \varpi (a_{2m-1}, a_{2m})\bigr)\]
gives a normalized even cochain. We can write $a_0da_1da_2\mapsto \rho (a_0)\varpi (a_1,a_2)$.
\end{ex}


The investigation of cocycles on algebraic curves was raised by Douglas and Yan \cite[p.~212]{DY}, and we address this in the following proposition. One can deduce this result from \cite[Theorem~5, p.~75/291]{Co} or (1.5) of \cite{Qi}, and we give a more explicit formula (\ref{cocycleformula}) for the cocycle.

\begin{prop}\label{Toeplitzcocycle}  Let $\calD = \calD_{\partial \Db}$ be the subalgebra of $L^\infty(\partial \Db; \Cb)$ as in Definition \ref{Dirichletalgebra}. 
\begin{enumerate}[(i)]
\item The Toeplitz map $\rho: \calD\to \calL(L^2(0,\infty);\Cb)$, $f \mapsto T_f$ gives a homomorphism modulo ${\calL}^1(L^2(0, \infty );\Cb )$, and the associated curvature is a product of Hankel integral operators
  \begin{equation}\label{semicommutator} \varpi (f_1,f_2)=\rho(f_1f_2)-\rho (f_1)\rho (f_2)=\Gamma_{\bar f_1}^\dagger\Gamma_{f_2}.\end{equation}
  \item The range of $\varpi: \Omega^1\calD\rightarrow \calL^1( (L^2(0, \infty ))$ is the set of trace-class integral kernels $K\in L^2((0, \infty )\times (0, \infty ); \Cb )$ such that ${\frac{d}{dt}}K(x+t,y+t)$ is a finite rank kernel, namely a product of vectorial Hankel integral operators.
\item There is a cyclic $1$-cocycle
   \begin{equation}\label{cocycleformula}
     \varphi_1(f_1,f_2) =
       \frac{1}{2\pi i}  \iint_\Db \frac{\partial (f_1,f_2)}{\partial (x,y)} \, dxdy           \qquad (f_1,f_2\in \calD ).
       \end{equation}
\end{enumerate}
\end{prop}

\begin{proof} (i) There is a natural unitary isomorphism $H^2(\Db )\rightarrow H^2(RHP)$ given by $f(z)\mapsto (\zeta +1)^{-1}f((\zeta -1 )/(\zeta +1))$ for $\Re \zeta >0$. Then there is a unitary isomorphism $H^2(RHP) \rightarrow L^2((0, \infty ); \Cb )$, 
$f\mapsto (2\pi )^{-1}\int_{-\infty }^\infty e^{i\omega t}f(i\omega ) \,d\omega $ as in the Paley--Wiener Theorem. The composition of conjugation by these unitary maps takes $\Gamma_h \in {\calL}(H^2(\Db ))$ to a Hankel integral operator in ${\calL}(L^2((0, \infty );\Cb ))$, which belongs to ${\calL}^2((L^2(0, \infty ); \Cb))$ whenever $\Gamma_h\in {\calL}^2(H^2(\Db ))$. 

Each $h\in \calD_{\Db}$ determines a harmonic function $h: \Db\rightarrow \Cb$ via the Poisson integral; then we can introduce a harmonic function $h((\zeta -1)/(\zeta +1))$ on $\{ \zeta : \Re \zeta >0\}$; note that
  \begin{equation}\label{DirichletRS}
\iint_{\Rb\times \Rb}{\frac{\vert h({\frac{iw-1}{iw+1}})- h({\frac{iv-1}{iv+1}})\vert^2}{\vert w-v\vert^2}} \,dvdw
    = \iint_{[0, 2\pi ]\times [0,2\pi]} 
    {\frac{\vert h(e^{i\theta })-h(e^{i\phi })\vert^2}{\vert e^{i\theta} -e^{i\phi }\vert^2}}\, d\theta d\phi .
  \end{equation}
  
(ii)
As in Proposition \ref{diagonaldiff} have 
\[(f_1,f_2)\mapsto \trace\bigl( \Gamma_{\bar f_1}^\dagger \Gamma_{f_2}\bigr) = \sum_{\ell=0}^\infty (\ell+1) {\hat f}_1(\ell) {\hat f}_2(-\ell).\]
In particular, each $h\in {\calD}_{\Db}$ produces a Hilbert-Schmidt Hankel operator via this route, so the semi-commutator is trace class.

(iii) By (i), we can introduce
   \begin{equation}
   \varphi_1(f_1,f_2) = \trace ( \varpi (f_1,f_2)-\varpi (f_2,f_1)),
   \end{equation}
so $\varphi$ is well defined and we clearly have $\varphi_1(f_2,f_1)=-\varphi_1 (f_1,f_2)$. We can extend $f,h\in{\calD}_\Db$ to harmonic functions on $\Db$ via the Poisson integral so that $\nabla f,\nabla h$ are square integrable with respect to area measure, so the Jacobian $\partial (f,h)/\partial (x,y)$ is integrable by the Cauchy--Schwarz inequality. 
Note that by \cite[p.~151]{HH} or (\ref{semicommutator}),
  \begin{align}\label{cocycleformula2}
  \trace(\varpi (f,h)-\varpi (h,f))
      &= \trace\bigl( [\rho (f), \rho (h)]\bigr)  \nonumber\\
      &= {\frac{1}{2\pi i}} \iint_\Db {\frac{\partial (f,h)}{\partial (x,y)}} \, dxdy
  \end{align}
and the identity (\ref{Jacobianidentity}). Hence the Hochschild coboundary operator is
  \begin{align}
  (b\varphi_1)(f_1,f_2,f_3)
    &= \trace\bigl( \varpi (f_1f_2,f_3)-\varpi (f_3,f_1,f_2)
          +\varpi (f_1,f_2f_3) \nonumber\\
    &\qquad -\varpi (f_2f_3,f_1)+\varpi (f_3f_1,f_2)
            -\varpi (f_2,f_3f_1)\bigr)   \nonumber\\
    &= \trace\bigl(\rho (f_3)\rho (f_1f_2)-\rho (f_1f_2)\rho (f_3)
            +\rho (f_1)\rho (f_2f_3)  \nonumber\\
    &\qquad -\rho (f_2f_3)\rho (f_1) +\rho (f_2)\rho (f_1f_3)
            -\rho (f_1f_3)\rho (f_2)\bigr) \nonumber\\
    &=0.  
    \end{align}
    This result justifies introducing $\calD$, since the integral in the cocycle formula converges absolutely only for $f_1,f_2\in \calD$. 
\end{proof}



\section{Linear systems and Fredholm determinant formulas}\label{S:LinearsystemsFredholm}

Let $(-A,B,C)$ be a continuous time linear system with state space $H$ and input and output space $\Cb$.  In Theorem~\ref{Lyapunovthm} below we introduce an operator algebra on the state space that captures most of the essential information about the linear system and shows how the ODE (\ref{linsystemsODE}) gives rise to a differential equation in this operator algebra. Under further hypotheses, we use Theorem \ref{Lyapunovthm} in Corollary \ref{crypto} to define an essential spectrum $\Sb_e$ for a linear system. In Lemma \ref{bracketlem}, we introduce a special differential ring which relates the algebras in Theorem \ref{Lyapunovthm} to differential equations on the real line. 

\begin{defn}\label{Lyapunovdef} 
\begin{enumerate}[(i)]
\item  Let $R_x$ be a family of operators such that $\langle R_xf,h\rangle$ is differentiable for all $f\in \calD (A)$ and $h\in \calD (A^\dagger )$. Then 
Lyapunov's equation for $(-A,B,C)$ is 
  \[ {\frac{dR_x}{dx}} = -A R_x - R_x A, \quad x > 0; \qquad 
           \Bigl({\frac{dR_x}{dx}}\Bigr)_{x=0} = -BC.\]
\item Let $\calE$ be a complex subalgebra of $\calL (H)$ with commutator subspace $[{\calE}, {\calE}]$. If $[\calE , \calE] \subset \calL^1(H)$ and $\calE$ is self-adjoint, then $\calE$ is said to be crypto-integral of dimension one; see \cite{HH}. (Note that all higher order commutators $[[U,V],W]$ have zero trace for $U,V,W\in \calE$.) 
\item Suppose that $[(\lambda I+A)^{-1}, (\mu I+A^\dagger )^{-1}]\in \calL^1(H)$. Then we say that $A$ and $A^\dagger$ are resolvent commuting modulo $\calL^1(H)$, or almost resolvent commuting.
\end{enumerate}
\end{defn}

\begin{thm}\label{Lyapunovthm} 
Suppose that $A$ has dense domain $\calD (A)$ and is maximally accretive. Suppose also that $C(sI+A)^{-1}\in H^2(RHP;H')$ and $(sI+A)^{-1}B\in H^2(RHP; H)$. Let $\calE_0$ be the unital complex subalgebra of $\calL (H)$ that is generated by $I, BC,$ and $(A-\lambda  I)(A+\lambda I)^{-1}$ for all $\lambda >0$, and let $\calE$ be the closure of $\calE_0$ in the weak operator topology. 
\begin{enumerate}[(i)]
\item Let $\calA_0=\calE_0/(\calE_0\cap \calL^1(H))$. Then $\calA_0$ is a commutative unital complex algebra, and there is an exact sequence of algebras 
  \begin{equation*}
 0  \xlongrightarrow{}  \calL^1(H)
    \xlongrightarrow{} \calE_0+\calL^1(H)
    \xlongrightarrow{} \calA_0
    \xlongrightarrow{} 0.
\end{equation*} 
\item Let $\varphi_1: {\calE}_0\times {\calE}_0\rightarrow \Cb$ be the bilinear form $\varphi_1 (V,W)= \trace ([V,W])$. Then $V\mapsto \varphi_1 (V,W)$ gives a trace on ${\calE}_0$ for all $W\in {\calE}_0$, and $\varphi_1 $ is a cyclic $1$-cocycle such that
  \begin{equation}\label{linPincus} 
  \varphi_1(V,W) = \log\det \bigl( e^{V}e^We^{-V}e^{-W}\bigr) 
                           \qquad (V,W\in \calE_0).
  \end{equation}
\item Also $\calE$ contains the family of trace class operators 
  \begin{equation}\label{Rdef}
     R_x = \int_x^\infty e^{-tA}BCe^{-tA} \,dt\qquad x\ge0,
  \end{equation}
and $(R_x)_{x\geq 0}$ satisfies Lyapunov's equation. 
\end{enumerate}
\end{thm}

\begin{proof} (i) Since $\calL^1(H)$ is an ideal in $\calL(H)$, one can check that all commutators of elements of $\calE_0$ belong to $\calL^1(H)$, so the commutator subspace of $\calE_0+\calL^1(H)$ is contained in $\calL^1(H)$.

(ii) A trace on ${\calE}_0$ is equivalent to a linear functional on the commutator quotient space 
  ${\calE}_0/[{\calE}_0, {\calE}_0]$. 
Let ${\calM}$ be an ${\calE}_0$-bimodule, and 
  $\delta : {\calE}_0\rightarrow \Der({\calM})$, $V \mapsto \delta_V$ where $\delta_V$ is the inner derivation $\delta_V(M) = VM-MV$. 
Then the inner derivations give a Lie algebra since $[\delta_V, \delta_W] = \delta_{[V,W]}$. By (i), $\varphi_1 (V,W)= \trace ([V,W])$ is well defined. Let ${\Rad}(\varphi_1 ) = \{ V\in {\calE}_0 : \varphi_1 (V,M) =0 \quad \forall M\in {\calE}_0\}$. Then for $V,W\in {\calE}_0$, we have 
$[V,W]\in {\calL}^1(H)$, hence $\trace ( [V,W]M-M[V,W]) = 0$ for all $M\in {\calL}(H)$; that is $[{\calE}_0, {\calE}_0] \subseteq {\Rad}(\varphi_1 )$.
We have $\varphi_1 (V,W)=-\varphi_1 (W,V)$ and 
   \begin{multline*} 
     \varphi_1 (UV,W)-\varphi_1 (U,VW)+\varphi_1 (WU,V)    \\
          = \trace\bigl( UVW-WUV-UVW+VWU+WUV-VWU\bigr)
          = 0.
    \end{multline*}
The formula (\ref{linPincus}) is valid by Pincus's identity \cite[p.~182]{HH0} since $[V,W]\in \calL^1(H)$. 

(iii) By the Lumer--Phillips theorem \cite{Go}, $-A$ generates a strongly continuous contraction semigroup $(e^{-tA})_{t\geq 0}$ on $H$, and $(sI+A)^{-1}\in {\calL}(H)$ for all $\Re s>0$. For $\lambda > 0$ let 
$(-A)_\lambda =-(\lambda /2) -(\lambda /2)(A-\lambda I)(A+\lambda I)^{-1}$. We note that each $(-A)_\lambda$ belongs to $\calE_0$ and $(-A)_\lambda f\rightarrow -Af$ in $H$ as $\lambda\rightarrow\infty$ for all $f\in \calD (A)$. Also $(\exp ( (-A)_\lambda t))_{t\geq 0}$ gives a uniformly continuous contraction semigroup on $H$, since $(-A)_\lambda \in \calL (H)$. Hence $e^{-tA}BCe^{-tA}\in\calE$ is an operator of rank one, and hence trace class, with $\Vert e^{-tA}BCe^{-tA}\Vert_{{\calL}^1(H)}\leq \Vert e^{-tA}B\Vert_H\Vert Ce^{-tA}\Vert_{H'}$. 

As in \cite[Proposition 2.1]{B1}, we have $C(sI+A)^{-1}=\int_0^\infty e^{-st}Ce^{-tA}\, dt$ and $(sI+A)^{-1}B=\int_0^\infty e^{-st}e^{-tA}B\, dt$, where by Plancherel's formula for Hilbert-space functions
  \begin{align*} 
  \int_0^\infty \Vert Ce^{-tA}\Vert_{H'}^2\, dt 
     &= \lim_{\varepsilon\rightarrow 0+} \int_{-\infty}^\infty \Vert 
             C((\varepsilon +i\omega )I+A)^{-1}\Vert^2_{H'} \,{\frac{d\omega}{2\pi}}, 
                              \nonumber\\
   \int_0^\infty \Vert e^{-tA}B\Vert_H^2\,dt 
      &= \lim_{\varepsilon\rightarrow 0+}\int_{-\infty}^\infty \Vert 
            ((\varepsilon +i\omega )I+A)^{-1}B\Vert^2_{H}\,{\frac{d\omega}{2\pi}},
    \end{align*}
and the right-hand sides are finite by hypothesis. Hence by the Cauchy--Schwarz inequality
   \begin{equation}
    \int_0^\infty \Vert e^{-tA}BCe^{-tA}\Vert_{{\calL}^1(H)}\, dt
        \leq \Bigl(\int_0^\infty \Vert e^{-tA}B\Vert^2_H\, dt\Bigr)^{1/2}
        \Bigl( \int_0^\infty \Vert Ce^{-tA}\Vert^2_{H'}\, dt\Bigr)^{1/2};
    \end{equation}
so $R_0$, and likewise $R_x$, define trace class operators in $\calE_0$. The operator $-A^\dagger$ generates the strongly continuous semigroup $(e^{-tA^\dagger})_{t\geq 0}$ of contractions on $H$ by \cite[Theorem 4.3]{Go}. For $f\in \calD (A^\dagger)$ and $h\in {\calD}(A)$, the expression 
$\ip< R_xh,f > = \int_x^\infty \langle BCe^{-tA}h, e^{-tA^\dagger}f \rangle \, dt$ is differentiable, and one easily checks using the fundamental theorem of calculus that Lyapunov's equation holds.
\end{proof}

\begin{rem} (1) We have used the ideal $\calL^1(H)$ since the Fredholm determinant is defined on $\{ I+\Gamma : \Gamma \in \calL^1(H)\}$ as we use in Lemma~\ref{Levitan} below. A similar result to Theorem~\ref{Lyapunovthm}(i) holds with the ideal of compact operators $\calK$ in place of $\calL^1(H)$, so $(\calE_0+\calK)/\calK $ is a commutative subalgebra of the Calkin algebra $\calL(H)/\calK$. This is of interest when $A$ has essential spectrum, and brings us within the scope of Proposition \ref{Tomi}.

(2) Let $(e^{-tA})_{t\geq 0}$ be the shift semigroup $(S_t)_{t\geq 0}$ as in Definition ~\ref{shift}. Then for suitable $B$ and $C$, the hypotheses and hence conclusions of Theorem~\ref{Lyapunovthm}(i),(ii),(iii) hold.

(3) To deal with self-adjoint Schr\"odinger operators, as in section~\ref{S:Spectral}, we use the following variant, and refer to \cite{B1} for examples relating to scattering theory. Corollary \ref{crypto} introduces an essential spectrum for linear systems that satisfy some conditions relating to self-adjointness. Examples \ref{Howlandex}(ii) and \ref{scatteringex} relate to Corollary \ref{crypto} (i) and (ii), while in Corollary \ref{cryptotrace} we obtain a determinant formula for multiplicative commutators relating to Corollary \ref{crypto}(iii).
\end{rem}

\begin{cor}\label{crypto} 
Suppose that both the linear systems $(-A,B,C)$ and $(-A^\dagger,C^\dagger,B^\dagger)$ satisfy the hypotheses of Theorem~\ref{Lyapunovthm}, and that either
\begin{enumerate}[(i)] 
\item $C=B^\dagger$ and $A$ is self-adjoint and non-negative; or
\item $C=B^\dagger$ and $iA$ is self-adjoint; or more generally
\item $A$ and $A^\dagger$ are resolvent commuting modulo $\calL^1(H)$. 
\end{enumerate}
Let $\calE_*$ be the unital complex subalgebra of $\calL (H)$ that is generated by $I, BC, C^\dagger B^\dagger$ and $(A-\lambda  I)(A+\lambda I)^{-1}$ and $(A^\dagger -\mu I)(A^\dagger+\mu I)^{-1}$ for all $\lambda, \mu  >0$. Then 
\begin{itemize} 
 \item $\calE_*$ is crypto-integral of dimension one;
 \item the weak closure of $\calE_*$ is a von Neumann algebra;
 \item the quotient of the weak closure of $\calE_*$ by the compact operators gives a commutative $C^*$ algebra with spectrum $\Sb_e$, so there exists a compact set $\Sb_e$ and a unital $*$-homomorphism $\calE_*\rightarrow C(\Sb_e ;\Cb )$ with dense range. In case (i), $\Sb_e$ is a subset of $[0, \infty )$; whereas in case (ii), $\Sb_e$ is a subset of the unit circle $\Sb^1$. 
\end{itemize}
\end{cor}

\begin{proof} (i) This is a special case of Theorem~\ref{Lyapunovthm}, in which $\phi (t)=Ce^{-tA}B$ is real. Hence $\calE_0=\calE_*$ is a self-adjoint subalgebra of $\calL (H)$ such that the commutator subspace is contained in $\calL^1(H).$ The canonical quotient map $\pi:{\calL}(H)\rightarrow {\calL}(H)/{\calK}(H)$ to the Calkin algebra restricts to a map on $\calE_*$ which is zero on $[\calE_*, \calE_*]$. Hence there is commutative $C*$-algebra $\calC$ such that $\pi (\calE_*)$ is dense in $\calC$. By the Gelfand-Naimark theorem, $\calC$ is $*$-isomorphic to $C(\Sb_e ;\Cb )$ for some compact set $\Sb_e$. Note that $BC$ is compact, hence $\pi (BC)=0$; also the spectrum of $(\lambda I-A)(\lambda I+A)^{-1}$ is contained in $[0, \infty )$ for all $\lambda >0$; hence the essential spectrum of $(\lambda I-A)(\lambda I+A)^{-1}$ is contained in $[0, \infty )$ for all $\lambda >0$; hence $\Sb_e\subset [0, \infty ).$

(ii) This is a special case of Theorem~\ref{Lyapunovthm}, in which $(e^{itA})_{t\in\Rb}$ is a unitary group and $\phi (t)=Ce^{-tA}C^\dagger$ has $\bar \phi (t)=\phi (-t).$ Here $\calE_*$ is a self-adjoint subalgebra of $\calL (H)$ such that the commutator subspace is contained in $\calL^1(H)$. Here the spectrum of $(\lambda I-A)(\lambda I+A)^{-1}$ is contained in $\Sb^1$, hence the essential spectrum of $(\lambda I-A)(\lambda I+A)^{-1}$ is contained in $\Sb^1$, so $\Sb_e\subseteq \Sb^1$.

(iii) The proof is as for Theorem~\ref{Lyapunovthm}. The key  identity is
\begin{equation}\bigl[ (A-\lambda I)(A+\lambda I)^{-1}, (A^\dagger -\mu I)(A+\mu I)\bigr]=4\lambda\mu \bigl[(A+\lambda I)^{-1}, (A^\dagger +\mu I)^{-1}\bigr].\end{equation}
\end{proof}

In section \ref{S:HowlandFredholm}, we obtain various results in the context of Corollary \ref{crypto}(i); Example \ref{scatteringex} relates to Corollary \ref{crypto}(ii). Proposition \ref{cryptotrace} applies in the context of Corollary \ref{crypto}(iii).


\section{Howland operators and Fredholm determinant formulas}\label{S:HowlandFredholm}
A Howland operator in an integral operator of the form
  \[ Kf(x) = \int_0^\infty \frac{b(x) c(y)}{x+y} f(y) \, dy,
     \qquad f \in L^2((0, \infty );\Cb )\]
where $b,c\in L^\infty ((0, \infty ); \Cb )$.  Howland observed \cite[page 410]{How} that some tools from the theory of Schr\"odinger operators could be adapted to study such integral operators and that Lyapunov's equation is fundamental. In our context, Howland operators provide an important situation where the conditions of Corollary~\ref{crypto}(i) hold, and such operators arise from linear systems of a special form which we now describe. In Proposition \ref{Levitan}, we compute Fredholm determinants via the Gelfand-Levitan equation.

\begin{ex}\label{Howlandex}
(i) 
Let $b,c\in L^\infty ((0, \infty );\Cb )\cap L^2((0, \infty );\Cb )$. Consider the linear system $(-A,B,C)$ with state space $H=L^2((0, \infty ); \Cb)$ where $\calD (A)=\{ f\in L^2((0, \infty ); \Cb ): xf(x)\in L^2((0, \infty );\Cb )\}.$ The operators are
  \begin{align}\label{howlandABC} 
    A: \calD (A)\rightarrow L^2((0, \infty);\Cb ),
       &\qquad f(u)\mapsto uf(u)\qquad (u > 0);\nonumber\\
    B:\Cb \rightarrow L^2((0, \infty );\Cb ),
       & \qquad x\mapsto bx\qquad (x\in \Cb);\nonumber\\
    C: L^2((0, \infty );\Cb )\rightarrow \Cb,
       & \qquad f\mapsto \int_0^\infty c(u)f(u)\, du.
  \end{align}
The impulse response function is given by the Laplace transform of $bc$, 
   \[\phi(t) = C e^{-tA} B = \int_0^\infty b(u)c(u)e^{-tu} \,du,\]
while $R_x = \int_x^\infty e^{-tA}BCe^{-tA} \,dt$ is the integral operator on $L^2((0, \infty );\Cb )$ that has kernel
\[ {\frac{ b(u)c(v)e^{-(u+v)x}}{u+v}}\qquad (u,v>0).\]
Thus $R_x$ is a Howland operator for all $x > 0$. 
If $c(u),b(u), c(u)/\sqrt{u}, b(u)/\sqrt{u}$ all belong to $L^2((0, \infty );\Cb)$, then the hypotheses of Theorem \ref{Lyapunovthm} are satisfied, and $R_0\in {\calL}^1(H)$. This was stated in \cite[Proposition 2.1]{B2}; here we give the crucial estimate
  \[ \Vert C((\varepsilon +i\omega )I+A)^{-1}\Vert^2_{H'}=\int_0^\infty {\frac{\vert c(u)\vert^2}{\vert \varepsilon +i\omega +u\vert^2}}du,\]
which gives
  \begin{align}
  \int_{-\infty}^{\infty} \Vert C((\varepsilon +i\omega )I+A)^{-1}\Vert^2_{H'} {\frac{d\omega}{2\pi}}
     &= \int_0^\infty \vert c(u)\vert^2
         \int_{-\infty}^{\infty} {\frac{1}{\vert \varepsilon +u+i\omega \vert^2}}{\frac{d\omega}{2\pi}} du\nonumber\\
     &={\frac{1}{2}} \int_0^\infty {\frac{\vert c(u)\vert^2 du}{u+\varepsilon}}.
  \end{align}
Thus $C(sI+A)^{-1}$ belongs to $H^2(RHP; H')$. A similar calculation shows that $(sI+A)^{-1}B \in H^2(RHP; H)$.

(ii) Suppose that $b=\bar c$; then $C=B^\dagger$, and Corollary \ref{crypto}(i) applies. 
(iii) The following Proposition \ref{Volterraprop} introduces a Volterra type group and applies to the corresponding $T_{GL}$ from (\ref{howlandABC}).

\end{ex}

For $a>0$, we have an orthogonal decomposition $L^2((0, \infty ); \Cb )=L^2((0,a);\Cb )\oplus L^2((a, \infty ); \Cb )$ where $L^2((a,\infty ); \Cb)$ is the image of the shift $S_a$ and  $L^2((0,a); \Cb )$ is the nullspace of $S_a^\dagger$.
Consider the space $\calT$ of all measurable $T:(0, \infty) \times (0,\infty )\rightarrow \Cb$ such that $T(x,y)=0$ for all $0<y<x$ and
$\int_0^\infty \int_x^\infty \vert T(x,y)\vert^2 \, dydx$ is finite, so $T$ is upper triangular. 
Given $T \in \calT$ define the 
bounded linear operator 
$V_T:L^2((0, \infty );\Cb )\rightarrow L^2((0, \infty );\Cb)$ by
\begin{equation}\label{Volterra}
   V_T f(x) = f(x)+\int_x^\infty T(x,y) f(y)\, dy 
      \qquad (f\in L^2((0, \infty ); \Cb)).
\end{equation}
\begin{prop}\label{Volterraprop} Let 
  \[ \calG = \{V_T \st \text{$T \in \calT$ 
    and $V_T$ is invertible in $\calL (L^2(0, \infty );\Cb )$} \}.
          \]
Then
\begin{enumerate}[(i)]
\item  for all $V\in\calG$ and $a>0$, the operator $V$ restricted to $L^2((0, a); \Cb )$ has image in $L^2((0,a); \Cb )$;
\item $\calG$ forms a group under operator composition;\par
\item If $V\in \calG$ is unitary, then $V=I$.
\item Suppose that $T \in \calT$, that  $T(x,y)$ is continuous on $\{ (x, y) \in (0, \infty )^2: 0<x<y\}$, and there exists $\varepsilon >0$ such that $e^{\varepsilon (x+y)}T(x,y)$ is bounded. Then with such a $T$ the formula (\ref{Volterra}) gives an operator $V_T\in\calG$.
\end{enumerate}
\end{prop}

\begin{proof} (i) For all $0<a<x$, we have $\int_0^aT(x,y) \,dy=0$, which is equivalent to the condition that $\int_0^a T(x,y)f(y) \, dy=0$ for all $0<a<x$ and $f\in L^2((0, a); \Cb )$. Recall the shift operator $S_a$ and note that $S_aS_a^\dagger $ is the orthogonal projection $L^2((0, \infty );\Cb )\rightarrow L^2((a,\infty );\Cb )$ and by the preceding calculation $S_aS_a^\dagger V(I-S_aS_a^\dagger )=0$. \par
 (ii) One considers the equation $V(I+H)=I$, which is equivalent to $H=-V^{-1}(V-I)$, which is upper triangular.\par
(iii) By (ii), the condition $VV^\dagger =I$ requires the adjoint kernel $\overline{T(y,x)}$ also to be upper triangular, whereas we have $\overline{T(y,x)}=0$ for all $0<x<y$ which implies the adjoint kernel is lower triangular; so 
$T(x,y)=0$ for all $x,y>0$. \par
(iv) Suppose that $\vert T(x,t)\vert\leq Me^{-\varepsilon (x+y)}$; then by iterated integration over a simplex, one shows that
\begin{align}\int_{\{x<y_1<y_2<\dots <y_n<y\}} &\vert T(x,y_1)T(y_1,y_2)\dots T(y_n,y)\vert dy_1\dots dy_n\nonumber\\
&\leq M^{n+1}e^{-\varepsilon (x+y)}\int_{\{x<y_1<y_2<\dots <y_n<y\}} e^{-2\varepsilon (y_1+\dots +y_n)} dy_1\dots dy_n\nonumber\\
&\leq {\frac{M^{n+1} e^{-\varepsilon (x+y)}}{\varepsilon^n 2^n n!}}.\end{align}
Thus one shows that the series $\sum_{j=1}^\infty (-1)^jT^j$ converges in $\calL^2$, so the spectrum of $V_T$ is $\{ 1\}$ and $V_T$ is invertible with $V_T^{-1}-I\in\calT$.
\end{proof}
\begin{prop}\label{Levitan}
Let $(-A,B,C)$ be a continuous time linear system as in Theorem \ref{Lyapunovthm}, with impulse response function $\phi$, and let
   \begin{equation}\label{Tfunction} T_{GL}(x,y)=-Ce^{-xA}(I+R_x)^{-1} e^{-yA}B,\qquad x,y \ge 0. \end{equation}
Then $T_{GL}$ satisfies the Gelfand--Levitan equation
   \begin{equation}\label{gelfandlevitanequation} \phi (x+y)+T_{GL}(x,y)+\int_x^\infty T_{GL}(x,z)\phi (z+y)\, dz=0,\end{equation}
and $(d/dx)\log\det (I+R_x)= \trace  T_{GL}(x,x)$.
\end{prop}

\begin{proof} We recall from \cite[Theorem~2.5]{BN} the basic identity
\begin{multline*} 
  Ce^{-(x+y)A}B -Ce^{-xA}(I+R_x)^{-1}e^{-yA}B
   -Ce^{-xA}(I+R_x)^{-1}
      \int_x^\infty  e^{-zA}BCe^{-zA}e^{-yA}B \, dz
   \\
  =Ce^{-xA}\Bigl( I-(I+R_x)^{-1}-(I+R_x)^{-1}R_x\Bigr)e^{-yA}B
  =0.
\end{multline*}
Also, the determinant satisfies
\begin{align}{\frac{d}{dx}}\log\det (I+R_x)&={\frac{d}{dx}} \trace \log (I+R_x)\nonumber\\
&=\trace \Bigl( (I+R_x)^{-1}{\frac{dR_x}{dx}}\Bigr)\nonumber\\
&=-\trace\bigl( (I+R_x)^{-1}e^{-xA}BCe^{-xA}\bigr)\nonumber\\
&=-Ce^{-xA}(I+R_x)^{-1}e^{-xA}B,\end{align}
as required.
\end{proof}

We can apply Proposition \ref{Levitan} directly in the context of Corollary~\ref{crypto}(i). To deal with Hankel products, we follow the approach of \cite{BD} and obtain a formula for the Fredholm determinant and its derivative from Proposition \ref{diagonaldiff}(vi). Note that if $\phi $ is the impulse response function of $(-A,B,C)$, then $\phi^\dagger$ is the impulse response function of $(-A^\dagger ,C^\dagger B^\dagger )$. 
Let $(-A,B_1,C_1)$ and $(-A,C_2,B_2)$ be linear systems with state space $H$ an impulse response functions $\phi (t)^\top =C_1e^{-tA}B_1$ and $\psi (t)=C_2e^{-tA}B_2$. To combine these, we let $\lambda\in \Cb$ be a complex parameter and let $(-\hat A, \hat B, \hat C)$ be the linear system
  \begin{equation}\label{matrixlinsystem} 
    (-\hat A, \hat B, \hat C)
     = \Biggl(\begin{bmatrix} -A&0\\ 0&-A\end{bmatrix},
            \begin{bmatrix} 0&B_1\\ B_2&0\end{bmatrix}, 
            \begin{bmatrix} \lambda C_1&0\\ 0&-C_2\end{bmatrix} \Biggr) \end{equation}
which has impulse response function
  \[ \Phi (x)
    = {\hat C}e^{-x{\hat A}}{\hat B}
    = \begin{bmatrix} 0&\lambda C_1e^{-xA}B_1\\ -C_2e^{-xA}B_2&0\end{bmatrix}
    = \begin{bmatrix} 0&\lambda\phi(t)^\top\\ -\psi (x)&0\end{bmatrix};\]
we momentarily defer the proof of existence. We introduce the operator function
  \begin{align} 
  \hat R_x
    =\int_x^\infty e^{-s{\hat A}}{\hat B}{\hat C}e^{-s{\hat A}} \,ds
    =\int_x^\infty 
         \begin{bmatrix} 
          0               &-e^{-sA}B_1C_2e^{-sA}\\ 
          \lambda e^{-sA}B_2C_1e^{-sA} & 0
          \end{bmatrix}\,ds .
\end{align}
\begin{rem}
It is straightforward to produce a suitable linear system that has impulse response function equal to any given rational function. For $a\in LHP$ and $r\in \Nb$, we take the state space to be $L^2((0, \infty );\Cb) $ and input and output spaces $\Cb$. Let
  \begin{align}\label{basicrational} 
  e^{-tA}: L^2((0, \infty );\Cb )\rightarrow L^2((0, \infty); \Cb ):
                &\qquad f(u)\mapsto e^{-tu}f(u)\qquad (u,t>0);
                                                \nonumber\\
  B:\Cb \rightarrow L^2((0, \infty );\Cb ):
                & \qquad x\mapsto u^{(r-1)/2}e^{au/2} x\qquad (x\in \Cb);
                                   \nonumber\\
  C: L^2((0, \infty );\Cb )\rightarrow \Cb:
        & \qquad f \mapsto {\frac{1}{(r-1)!}}\int_0^\infty u^{(r-1)/2} e^{au/2}f(u)\, du.
  \end{align}
Then the impulse response function is 
  \[\phi (t) = Ce^{-tA}B
     ={\frac{1}{(r-1)!}}\int_0^\infty u^{r-1}e^{-tu+au} \,du
     ={\frac{1}{(t-a)^r}},\]
while $R_x=\int_x^\infty e^{-tA} B C e^{-tA} \,dt$ is the integral operator on $L^2((0, \infty );\Cb )$ that has kernel
\[ R_x\leftrightarrow {\frac{ (uv)^{(r-1)/2)}e^{(a/2-x)(u+v)}}{u+v}}\qquad (u,v>0),\]
which has the form of a Howland operator as in Example \ref{Howlandex} and \cite{How}; here $1/(u+v)$ is the Carleman kernel as in the proof of Lemma \ref{boundedHankel}, while the other factors arise from multiplication operators.

By an elaboration of this example, one can introduce (\ref{matrixlinsystem}) to realize the  impulse response functions as in 
(\cite[Proposition 3.2]{BD}). Given linear systems $(-A,B_1,C_1)$ with impulse response function $\phi_1$ and state space $H$, and $(-A,B_2,C_2)$ with impulse response function $\phi_2$, then 
  \begin{equation} \Bigl( \begin{bmatrix} -A&0\\ 0&-A\end{bmatrix}, \begin{bmatrix} B_1\\B_2\end{bmatrix}, \begin{bmatrix}C_1&C_2\end{bmatrix}\Biggr)
  \end{equation}
has impulse response function $\phi_1+\phi_2$ and state space $H\otimes \Cb^2$. Thus one can realise rational function via partial fractions by adding the impulse response functions from sytsems as in (\ref{basicrational}).
\end{rem}

\begin{prop}\label{GelfandLevitanProp} 
Let $K = \Gamma_\phi^\top \Gamma_\psi \in \calL^2$ be a product of vector-valued Hankel operators and let $K_t = S_t^\dagger K S_t$ be as in Proposition \ref{diagonaldiff}(vi).
Suppose that $A$ has dense domain $\calD (A)$ and is maximally accretive. Suppose also that $C_j(sI+A)^{-1}\in H^2(RHP;H')$ and $(sI+A)^{-1}B_j\in H^2(RHP; H)$ for $j=1,2$ and that $(-\hat A, \hat B, \hat C)$ is defined as in (\ref{matrixlinsystem}). Then 
  \begin{equation} 
  \hat T_{GL}(x,y) = -{\hat C}e^{-x{\hat A}} 
                     (I+{\hat R}_x)^{-1} e^{-y{\hat A}}{\hat B}.
  \end{equation}
satisfies the Gelfand--Levitan equation
  \begin{equation}\label{GelfandLevitan}
  {\hat \Phi}(x+y)+ {\hat T_{GL}}(x,y)
      + \int_x^\infty {\hat T_{GL}}(x,z) {\hat \Phi}(z+y)\, dz = 0, 
  \end{equation}
and satisfies
\begin{equation}
   \trace {\hat T_{GL}}(x,x) 
       = {\frac{d}{dx}} \log \det (I+\lambda K_x).
   \end{equation}
\end{prop}

\begin{proof}

See \cite[Lemma 5.1]{B1} for the Gelfand--Levitan equation. With $\phi_{(x)}(y)=\phi (y+2x)$, the Fredholm determinant satisfies
  \begin{align} \det (I+\lambda K_x)
  &=\det (I+\lambda\Gamma_{\phi_{(x)}^\top}\Gamma_{\psi_{(x)}})\nonumber\\
  &=\det (I+\Gamma_{\hat\Phi_{(x)}})\nonumber\\
  &=\det (I+\hat R_x).
  \end{align}
Now $\hat R_x$ satisfies $-{\frac{d\hat R_x}{dx}}=e^{-x\hat A}\hat B\hat Ce^{-x\hat A}$, so by Proposition 2.4 of \cite{B1}, we have
  \begin{align}
  {\frac{d}{dx}}\log\det (I+\hat R_x)
    &=\trace \Bigl( (I+\hat R_x)^{-1}{\frac{d\hat R_x}{dx}}\Bigr)\nonumber\\
    &=-\trace \bigl((I+\hat R_x)^{-1}e^{-x\hat A}\hat B\hat Ce^{-x\hat A}\bigr)\nonumber\\
    &=-\trace \bigl( \hat Ce^{-x\hat A}(I+\hat R_x)^{-1}e^{-x\hat A}\hat B \bigr),
  \end{align}
as required.
\end{proof}

Corollary \ref{crypto}(ii) applies to several problems about scattering for Schr\"odinger operators on the real line, as in the following.

\begin{ex} \label{scatteringex} We consider a skew symmetric $A$ as in Corollary \ref{crypto}(ii), and how one approximates a unimodular function on the spectrum. Let $b_1,b_1\in C_0(\Rb ;\Cb)$ such that $b_1(-k)=\overline{b_1(k)}$,  $b_2(-k)=\overline{b_2(k)}$, $\vert b_1(k)\vert =\vert b_2(k)\vert$ and let $b(k)=b_1(k)b_2(k)$, so $b(-k)=\overline{b(k)}$. As in \cite[Theorem 4.2]{B1} we consider the linear system $(-A,B,C)$ given by
\begin{align} e^{-xA}:L^2(\Rb ;\Cb )\rightarrow L^2(\Rb ;\Cb )&:\qquad f(k)\mapsto e^{-ikx}f(k);\nonumber\\
B: \Cb\rightarrow L^2( \Rb; \Cb )&\qquad  a\mapsto b_1(k)a;\nonumber\\
C: L^2(\Rb; \Cb )\rightarrow \Cb: & \qquad f\mapsto {\frac{1}{2\pi}}\int_{-\infty}^\infty b_2(k)f(k) dk.\end{align}
Then $(e^{-xA})_{x\in \Rb}$ is a unitary group of operators, and $\sigma (A)=i\Rb$. For $\zeta>0$, the operator $(\zeta I-A)(\zeta I+A)^{-1}$ is unitary, being given by multiplication by a unimodular function.  
This example occurs in scattering problems \cite[Theorem 1.3 (3) and Theorem 4.2]{B1} and \cite[Section 5]{EM}. To convert our notation to that of \cite[p.~486]{EM}, let the left-hand scattering coefficient be $s_{21}(k)=b(k)$ and $\phi (x)=(2\pi )^{-1}\int_{-\infty}^\infty s_{21}(k) e^{ixk}dk$, so the Fredholm determinant is
\begin{equation}\label{scattertau}\vartheta (s_{21}) =\det (I+\Gamma_\phi )=\det (I+R_0).\end{equation}
Note that with $\phi_{(y)}(x)=\phi (x+2y )$, the backward shift $\phi (x)\mapsto \phi_{(y)}(x)$ arises from multiplication $b(k)\mapsto e^{2iky}b(k)$.\par
\indent Let $(\lambda_j)_{j=1}^\infty$ be a complex sequence such that $\Re \lambda_j>0$ for all $j$, and suppose that there exists $\delta_1>0$ such that
\begin{equation}\label{Carleson}\prod_{j=1; j\neq \ell}^\infty \Bigl\vert {\frac{\lambda_\ell-\lambda_j}{\lambda_\ell+\bar\lambda_j}}\Bigr\vert\geq \delta_1 \qquad (\ell =1, 2, \dots ); \end{equation}
then $(\lambda_j)_{j=1}^\infty$ satisfies Carleson's interpolation condition. Then the Blaschke product, with convergence factors,  
\begin{equation}\label{Blaschke} \beta_1(z)=\prod_{j=1}^\infty e^{i\alpha_k} {\frac{z-\lambda_j}{z+\bar\lambda_j}}; \qquad e^{i\alpha_j} {\frac{1-\lambda_j}{1+\bar\lambda_j}}\geq 0 \end{equation}
has simple zeros at $\lambda_j$ for $j=1, 2, \dots$, and is called an interpolating Blaschke product. One shows that for almost all $k\in \Rb$, there is a limit $\beta_1(ik)=\lim_{x\rightarrow 0+}\beta_1(x+ik)$ such that $\vert \beta_1(ik)\vert =1$. Jones \cite{J} refined the Douglas-Rudin approximation theorem and proved that for any measurable
function $U:i\Rb\rightarrow \Cb $ such that $\vert U(ik)\vert =1$ for all $k\in \Rb$, and $\varepsilon >0$, there exist interpolating Blaschke products $\beta_1$ and $\beta_2$ such that 
\begin{equation}\label{rationalapprox} \Bigl\vert U(ik)-{\frac{\beta_1(ik )}{\beta_2(ik)}}\Bigr\vert <\varepsilon \end{equation}
for almost all $k\in \Rb$. Hence there exists a sequence $(\Phi_N(z))_{N=1}^\infty$ of rational functions such that $\Phi_N(z)\rightarrow \beta_1(z)/\beta_2(z)$ as $N\rightarrow\infty$ for all $\Re z>0$ and $\vert \Phi_N(ik)\vert =1$ for all $k\in \Rb$ and $N=1,2, \dots $.
\end{ex}

\begin{rem}\label{BakerAkhiezer} We return to the differential equation of Example~\ref{Schrodingerasmatrix}. Choose $x_0>0$ such that $\Vert R_x\Vert<1$ for all $x>x_0$ and let $H=L^2( (x_0, \infty );\Cb )$. 
Consider
   \begin{equation}\label{BakerAk}
   f (x;\kappa ) = \cos \sqrt{\kappa} x+\int_x^\infty T_{GL}(x,y) \cos \sqrt{\kappa} y\, dy\qquad (x>x_0)
   \end{equation}
so that $-{\frac{d^2}{dx^2}}f (x;\kappa )+u(x)f (x;\kappa )=
\kappa f (x;\kappa )$ by \cite[Section 3]{BN}.

In section~\ref{S:Spectral}, we consider the family $(-A,(\zeta I+A)(\zeta I-A)^{-1}B,C)$ and use this to introduce the notion of a spectral curve. In \cite{DGP} and \cite[Proposition 5.14]{SW}, the authors consider a Baker--Akhiezer function $\varphi_W$ which is parametrized by a Grassmannian of subspaces of $H$ and is expressed as $\varphi_W (x)=K(e^{xz})$ where $K$ is a formal operator; our $I+T_{GL}$ replaces $K$.
\end{rem}


\section{The xi function and Darboux addition}\label{S:xi}

Following Kre{\u{\i}}n \cite{Kr}, Gesztesy and Simon \cite{GS} introduced the \textsl{xi} or spectral displacement function associated with a one-dimensional Schr{\"o}dinger operator, and described how this function connects the Green's function (\ref{Krein}), the potential and other information in the study of spectral and inverse spectral problems.

Let $L$ denote the densely defined self-adjoint operator $L=-{\frac{d^2}{dx^2}}+u$ in $L^2(\Rb ;\Cb )$ associated to a potential $u \in C_b(\Rb,\Rb)$. 
Suppose that $\psi_{+},\psi_{-}$ are solutions of $L\psi =\lambda \psi$ such that $\psi_+$ restricts to function in $L^2((0,\infty) ;\Cb )$ and $\psi_{-}$ restricts to a function in $L^2((-\infty,0) ;\Cb )$.

\begin{defn}
Kodaira's \cite{Ko} characteristic matrix is
\begin{equation}\label{Charmatrix} \Xi (x,\lambda )=\begin{bmatrix} 2\psi_+\psi_-&\psi'_+\psi_-+\psi_+\psi'_-\\
\psi'_+\psi_-+\psi_+\psi'_-&2\psi'_+\psi'_-\end{bmatrix}\end{equation}
all evaluated at $(x, \lambda )$.
\end{defn}
Then $\Xi (x, \lambda )$ is an entire function of $\lambda$ and $\Xi (x,\lambda )=\Xi (x,\lambda )^\top.$ 
The Green's function $G$ for $(\lambda I-L)^{-1}$ is
   \begin{equation}\label{GreenWronskian}  G(x,y;\lambda ) = 
      \begin{cases}
   \psi_-(x;\lambda )\psi_+(y,\lambda )/\Wr(\psi_-,\psi_+)  & (x<y) \\
   \psi_-(y;\lambda )\psi_+(x;\lambda )/\Wr(\psi_-, \psi_+) & (x>y).
    \end{cases}
    \end{equation}
For any $f\in L^2(\Rb, \Cb)$, the function
$h(x)=\int_{-\infty}^\infty G(x,y;\lambda )f(y) \,dy$ satisfies $f=\lambda h+h''-uh$.

We begin by considering the Darboux addition (\ref{Darbouxaddition}). The we consider the infinitesimal Darboux addition (\ref{infinitesimalDarboux}), which is directly related to the diagonal of the Green's function. 
Suppose that $(e^{-tA})_{{t\geq 0}}$ is a $(C_0)$ contraction semigroup, so $A$ has dense domain $\calD (A)$; that $A$ is accretive, so $\Re \langle Af,f\rangle \geq 0$ for all $f\in \calD (A)$; and that $(\lambda I+A)\calD (A)$ is dense in $H$ for some $\lambda>0$. Then for all $\zeta \in \Rb$, the operator $(\zeta I-A)(\zeta I+A)^{-1}$ is a contraction on $H$. 
 In previous papers \cite{BD} and \cite[p.~79]{BN}, we have considered the Darboux addition rule (\ref{Darbouxaddition}) on a continuous time linear system, namely
\begin{equation}\label{Darbouxaddition}
   (-A,B,C) \mapsto \Sigma_\zeta =(-A, (\zeta I+A)(\zeta I-A)^{-1}B,C)\qquad (\vert \arg (-\zeta )\vert <\pi /2).
\end{equation} 
As in \cite{BD}, we are particularly interested in the pair $\Sigma_\infty =(-A, B,C)$ and $\Sigma_0=(-A, -B,C)$, and regard the involution $(-A,B,C)\leftrightarrow (-A,-B,C)$ as analogous to the involution on a hyperelliptic curve, as in \cite{FK} or \cite{Mum}. 
We introduce the notion of a Darboux curve for a linear system, where the definition is motivated by \cite[3.4]{EM}.

\begin{defn}
Let $\calA_0$ be the unital complex subalgebra of $\calL (H)$ that is generated by the operators $(\lambda I-A)(\lambda I+A)^{-1}$ for $\lambda \in \Cb\setminus \sigma (-A)$, and let $\calA$ be the closure of $\calA_0$ in $\calL (H)$ for the operator norm topology. 
The Darboux curve of $(-A,B,C)$, denoted $\Sb$, is the maximal ideal space of the unital commutative Banach algebra $\calA$.  
\end{defn}

The Gelfand transform gives an algebra homomorphism $\calA\to C(\Sb ; \Cb )$, so elements of $\calA$ give functions on $\Sb$; clearly $\Sb$ is unchanged by Darboux addition. 
Note that $H$ is a Hilbert module over $\calA$, and $H$ is a Hilbert module over the complex rational functions with poles off $\sigma (A)$. One regards $(\zeta I+A)(\zeta I-A)^{-1}$ as a multiplier on $\calA$ or a rational function on $\Sb$, with inverse $(\zeta I+A)^{-1}(\zeta I -A)$. Thus Darboux addition corresponds to addition of divisors on $\Sb$. In the special cases covered by Corollary \ref{crypto} we introduced $\Sb_e$ which involves the essential spectrum of $A$, hence $\Sb_e$ is a closed subset of $\Sb$. The meaning of $\zeta$ will be revealed in Theorem~\ref{BurchnallChaundy} as a spectral parameter for a spectral curve $\Xb$. 

\begin{defn}  For $0<\theta \leq\pi$, we introduce the sector $S_\theta=\{ z\in \Cb \setminus \{0\}: \vert \arg z\vert <\theta\}$. A closed and densely defined linear operator $-A$ is a generator of type $\calG\calA_b(\theta ,M)$ \cite[Theorem 5.3]{Go} if there exists $\pi/2<\theta<\pi$ and $M\geq 1$ such that $S_\theta$ is contained in the resolvent set of $-A$ and 
$\vert\lambda \vert\Vert (\lambda I+A)^{-1}\Vert_{{\mathcal L}(H)}
    \leq M$ 
for all $\lambda\in S_\theta$. Let ${\calD}(A)$ be the domain of $A$ and ${\calD}(A^\infty )=\cap_{n=0}^\infty {\calD}(A^n)$. See \cite[p.~37]{Go}. 
\end{defn}

\begin{lem}\label{sectorial}
Let $(-A,B,C)$ be a linear
system such that $\Vert e^{-t_0A}\Vert_{{\mathcal L}(H)}<1$
for some $t_0>0$, and that $B$ and $C$ are Hilbert--Schmidt
operators such that $\Vert B\Vert_{{\mathcal L}^2(H_0;H)}\Vert C\Vert_{{\mathcal L}^2(H;H_0)}\leq 1$. 
Suppose further that $-A$ is of type $\calG\calA_b(\theta ,M)$.  
Then the 
trace class operators $(R_x)_{x>0}$ give the unique solution to Lyapunov's equation  (\ref{Lyapunov}) for $x>0$ that
satisfies the initial condition.
\end{lem}

\begin{proof} By \cite[Theorem 5.3]{Go}, there is a $(C_0)$ uniformly bounded semigroup $(e^{-tA})_{t>0}$ which has an analytic extension to the sector $\{ t\in \Cb : \Re t>0, \vert \arg t\vert <\theta -\pi/2\}$ so that for all $\varepsilon >0$ there exists $M_\varepsilon <\infty$ such that $\Vert e^{-tA}\Vert\leq M_\varepsilon$ for all $t\in \Cb $ such that $\Re t>0$ and $\vert \arg t\vert <\theta-\pi/2-\varepsilon$.  Then we can follow the proof of \cite[Theorem 2.3]{BN}.
\end{proof}

Suppose that $(-A,B,C)$ is a continuous time linear system with state space $H$ and input and output space $\Cb$ which satisfies the conditions of Lemma~\ref{sectorial}, so $C$ is scalar-valued. Let $\calE \subseteq L(H)$ be as in Theorem~\ref{Lyapunovthm}
\cite[Theorem 4.4]{BN}. Suppose that $I+R_x$ is invertible in $\calL (H)$ and define $F_x=(I+R_x)^{-1}$; then define the bracket operation 
$\lfloor\,\cdot \,\rfloor_x: C((0, \infty );\calE )\rightarrow C^\infty ((0, \infty ); \Cb )$
\begin{equation}\label{bracket} 
   \lfloor X\rfloor_x =
      Ce^{-xA}F_xXF_xe^{-xA}B
        \qquad (X\in C((0, \infty );\calE ),\, x>0),
\end{equation}
so that $\lfloor X\rfloor_x$ is a scalar-valued function of $x>0$. Here $(F_x)_{x>0}$ is a family of elements of ${\calE}$. Equipped with the special associative multiplication 
\begin{equation} \label{product}
  X\ast Y = X\bigl(AF_x
      + F_xA-2F_xAF_x\bigr)Y
        \qquad (X,Y\in C((0, \infty ); \calE ))
\end{equation}
and the formal first-order differential operator
\begin{equation}\label{deriv}
  \partial X = A(I-2(I+R_x)^{-1})X + {\frac{dX}{dx}} +  
                 X(I-2(I+R_x))^{-1}A \qquad (X\in C^\infty (0, \infty );\calE));
\end{equation}
then $C^\infty ((0, \infty );\calE )$ becomes a differential ring in sense of \cite{vdPS}.

\begin{lem}\label{bracketlem} 
The bracket operation  $\lfloor \, \cdot \, \rfloor_x : \calE \rightarrow C^\infty ((0, \infty ); \Cb )$ 
gives a homomorphism of differential rings, so that 
\begin{equation} \quad 
\lfloor X\ast Y\rfloor_x
    = \lfloor X\rfloor_x\lfloor Y\rfloor_x, \qquad
  \lfloor \partial X\rfloor_x 
    = {\frac{d}{dx}}\lfloor X\rfloor_x.
\end{equation}
\end{lem}

\begin{proof} This follows from Lyapunov's equation (\ref{Lyapunov}) as in \cite[Theorem 4.4]{BN}.
\end{proof}

This is a counterpart to P\"oppe's identities for Hankel integral operators, as in \cite[section 3.5]{McK2}. Using Lemma \ref{bracketlem}, one can move between operator identities in the algebra $\calE_0$ on $H$ to differential equations for $\Cb$-valued functions on $(0, \infty )$; see also \cite[Section 2]{BD}.

To associate a Schr{\"o}dinger operator with a linear system $(-A,B,C)$, we introduce the potential $u(x) = -4\lfloor A \rfloor_x$
, which also depends on the $B$ and $C$ via the bracket (\ref{bracket}). This is equivalent to  
 \begin{equation}\label{Dyson} 
 u(x)=-2{\frac{d^2}{dx^2}}\log\det (I+R_x)  
 \end{equation} 
as in Dyson's determinant formula; see \cite[(1.12)]{BN} and (\ref{scattertau}). By Proposition \ref{Levitan}, the corresponding Schr\"odinger operator is $L=-{\frac{d^2}{dx^2}}+u$.
The Darboux addition rule (\ref{Darbouxaddition}), $\zeta \mapsto \Sigma_\zeta$, produces a family of linear systems for which we can similarly associate bracket operations,  potentials $u_\zeta$ and Schr{\"o}dinger operators $L_\zeta = \frac{d^2}{dx^2} + u_\zeta$. The evolution of the potentials under this rule is of the form $u_\zeta = u - \zeta X(u) + o(\zeta)$ where $X(u)$ is a vector field on the space of potentials. Following McKean \cite{McK} we showed in \cite[Proposition~2.5]{BD} that under certain conditions, the infinitesimal Darboux addition $X(u)$ is completely determined by the diagonal Green's function.
Let $F_x=(I+R_x)^{-1}$ and introduce 
\begin{equation}\label{infinitesimalDarboux} 
   X(u) = {\frac{-2}{\sqrt{-\lambda }}} 
     \bigl\lfloor A(I-2F_x)A(\lambda I+A^2)^{-1}  
       + A(\lambda I+A^2)^{-1}(I-2F_x)A\bigr\rfloor_x. 
\end{equation}
as in \cite[Proposition 2.5]{BD}.

\begin{prop}\label{spectralprop} 
\begin{enumerate}[(i)]
\item The entries of the characteristic matrix $\Xi (x, \lambda )$ for $x>0$ belong to $\calE$.
\item Let $L$ be the Schr\"odinger operator in $L^2(\Rb ;\Cb )$, and let $\Xi (x, \lambda )$ be defined for $x\in \Rb$. Then $\Im \Xi (x, \lambda ) $ determines the spectral measure of $L$.

\end{enumerate}
\end{prop} 
\begin{proof} (i) We observe that $\det \Xi (x, \lambda )=-{\hbox{Wr}}(\psi_+, \psi_- )^2$, which is a non-zero constant, and the top-left entry of $\Xi$ is $\psi_+\psi_-$ as in the numerator of $-G(x,x;\lambda )$ given by (\ref{GreenWronskian}). 

Let $\lambda_0$ be the bottom of the spectrum of $L$. Then by \cite[(3.2)]{BD}, the diagonal Green's function for $(\lambda I-L)^{-1}$ is
\begin{equation}\label{diagonalgreen} 
  G(x,x;\lambda )
   ={\frac{1}{\sqrt{-\lambda}}}
      \Bigl({\frac{1}{2}}-{\frac{\lfloor A\rfloor_x}{\lambda}}+{\frac{\lfloor A^3\rfloor_x}{\lambda^2}}
       -{\frac{\lfloor A^5\rfloor_x}{\lambda^3}}+\dots \Bigr)
       \qquad (\lambda <-\lambda_0),
\end{equation}
where $\lambda$ is a spectral parameter for $L$.
Also $X(u)=-2{\frac{d}{dx}} G(x,x;\lambda )$ gives an element of $\calE$, as does 
${\frac{d^2}{dx^2}} G(x,x;\lambda )$. Correspondingly, the off-diagonal entry is $(\psi_+\psi_-)'=\psi_+'\psi_-
+ \psi_+\psi_-'$. Finally, the remaining entry is 
  \[ 2\psi_+'\psi_-'=(\psi_+\psi_-)''+2(\lambda -u)\psi_+\psi_-. \]
By Kodaira's theorem \cite{Ko}, for Schr\"odinger's operator in $L^2(\Rb ;\Cb )$, the spectral measure matrix is
  \begin{equation}\label{Kodaira} \mu (0,\lambda ]=\lim_{\varepsilon\rightarrow 0+}{\frac{1}{\pi}}\int_0^\lambda \Im \Xi (x, \nu +i\varepsilon ) \, d\nu .\end{equation}
Here $\mu (\nu, \lambda ]$ is a positive semi-definite matrix for $\lambda \geq \nu $, so one can form a positive matrix measure 
$d\mu (\lambda )$ as in Stieltjes integration.

(ii) Let $m_{\pm}(x,\lambda )$ be the Weyl functions for $Lf=-{\frac{d^2f}{dx^2}}+uf$ in $L^2(\Rb; \Cb )$ with boundary condition $f(x)=0$ as in (\ref{Weylm}), so that $f+m_+h\in L^2((0, \infty );\Cb )$ and $f-m_-h\in L^2((-\infty ,0);\Cb )$. Then $m_{\pm}(x, \lambda )$ are Herglotz functions that have boundary limits 
 \[ m_{\pm}(x,\lambda ) = \lim_{\varepsilon\rightarrow 0+} m_{\pm}(x,\lambda +i\varepsilon ). \] 
From \cite[(6.1) and (6.2)]{GW}, we have the identities
   \begin{align} 
   -G(x,x;\lambda )
      &={\frac{-1}{m_+(x,\lambda )+m_-(x, \lambda )}},\\
   -{\frac{d}{dx}}G(x,x;\lambda )
      &=-{\frac{m_+(x,\lambda )-m_-(x, \lambda )}{m_+(x,\lambda )+m_-(x, \lambda )}}, 
    \end{align}
and so 
  \begin{equation}\label{logdiffgreen} {\frac{d}{dx}}\log \bigl( -G(x,x;\lambda )\bigr) =m_+(x,\lambda )-m_-(x, \lambda ). \end{equation}
Suppose further that the spectrum is purely absolutely continuous. Then the diagonal Green's function determines the spectral measure as
  \begin{equation}\label{specdensityfromgreen} {\frac{d}{dx}}\log \bigl( -G(x,x;\lambda )\bigr) =2\pi i {\frac{d\mu }{d\lambda }}, \end{equation}
at Lebesgue points $\lambda$ of the spectral density, as in (\ref{specdensity}).
\end{proof}

\begin{ex} In Example \ref{scatteringex} and \cite[Theorem 4.2 and Section 6]{B1}, we have shown that the hypotheses of Proposition~\ref{spectralprop}(ii) hold for the scattering case of Schr\"odinger's equation on the real line. The linear systems $(-A,B,C)$ is associates with a unitary group $(e^{-xA})_{x\in\Rb}$ that is determined by the scattering data in \cite[(4.9)]{B1}.   
\end{ex}

Proposition \ref{spectralprop} amounts to a forward spectral theorem, which determines the spectral measure. The inverse spectral problem is also solved likewise, in the sense that the potential $u$ is determined by partial information about $G(x,x; \lambda )$.  

\begin{defn} Let $u\in C_b(\Rb; \Rb )$ and let $L$ be Schr\"odinger's operator $L\psi =-{\frac{d^2 \psi}{dx^2}}+u \psi$ with Green's function as in \ref{GreenWronskian}. 
\begin{enumerate}[(i)] 
\item The \textsl{xi} function of Kre{\u{\i}}n and Gesztesy--Simon \cite{GS} is defined to be 
  \begin{equation}\label{Krein}
    \xi (x,\lambda ) = {\frac{1}{\pi}} \arg 
       \lim_{\varepsilon\rightarrow 0+}
       \bigl( -G(x,x;\lambda+i\varepsilon ) \bigr)
          \qquad (\lambda\in \Rb ).
    \end{equation}
\item The potential $u$ is discretely dominated if $\xi (x,\lambda )=1/2$ holds for all $x\in \Rb$ and almost all $\lambda$ in the essential spectrum of $L$ with respect to Lebesgue measure.
\end{enumerate}
\end{defn}

Now $\log (-G(x,x;\lambda ))$ is a Herglotz function, hence has boundary values as in (\ref{Krein}). Observe that $\xi (x,\lambda )=1/2$ if and only if $G(x,x;\lambda )$ is purely imaginary. Suppose that $u$ is discretely dominated, that $L$ has spectrum $\sigma (L)$ with $\lambda_0=\inf\sigma (L)$ and
   \begin{equation}\label{gaps}
     [\lambda_0, \infty )\setminus \sigma (L)
       = \cup_{j=1}^\ell (\alpha_j, \beta_j),
   \end{equation}
so the spectrum has $\ell$ consecutive gaps $(\alpha_j, \beta_j)$. When $\ell$ is finite, the potential is finite-gap, or algebro-geometric; Theorem \ref{BurchnallChaundy} provides examples.

The notion of a \textsl{xi} function can further be generalized via Pincus's determining function to a bilinear trace formula \cite[Theorem 3.6]{CS}. The determining function, and related determinants, are more conveniently introduced via discrete-time linear systems.  We shall discuss this next, before we return to continuous-time systems at the end of the section.

\begin{defn}\label{discretelinsystem}
\begin{enumerate}[(i)]
  \item Let $U\in {\calL}(H)$, $B\in {\calL}(H_0,H)$, $C\in {\calL}(H,H_0)$, $D\in {\calL}(H_0)$. We denote by 
    $\begin{bmatrix} U&B\\ C&D\end{bmatrix}$ 
 the discrete-time linear system
  \begin{align*} 
   x_{n+1}  &= Ux_n+Bu_n  \\
   y_n      &= Cx_n+Du_n\qquad (n=0, 1, \dots )
 \end{align*}
with input $(u_n)_{n=0}^\infty$ with $u_n\in H_0$, state $(x_n)_{n=0}^\infty$ with $x_n\in H$ and output $(y_n)_{n=0}^\infty$ with $y_n\in H_0$. The transfer function for this system is $\Phi (\lambda )=D+C(\lambda I-U)^{-1}B$. 
  \item Let $V\in \calL (H)$. For
the family of linear systems 
  \begin{equation}\label{determiningsystem}
    \begin{bmatrix}
      U   &  (\mu I-V)^{-1}B\\ 
      iC  &   I
    \end{bmatrix}
     \qquad (\mu \in \Cb\setminus\sigma (V)),
  \end{equation}
Pincus \cite{Pi} defined the determining function
  \begin{equation}\label{determining} 
    E(\lambda , \mu ) = I + i C(\lambda I-U)^{-1} (\mu I-V)^{-1} B,
      \qquad (\lambda \in \Cb \setminus \sigma(U),
            \ \mu \in \Cb \setminus \sigma(V)).
  \end{equation}
\end{enumerate}
\end{defn}

The main application of the determining function is to compute determinants, starting with the identity 
\[ \det E(\lambda ,\mu )=\det (I+i(\mu I-V)^{-1}BC(\lambda I-U)^{-1})\qquad 
(\lambda \in \Cb \setminus \sigma(U),
            \ \mu \in \Cb \setminus \sigma(V)).\] 
In particular, let $U,V\in {\calL}(H)$ be self-adjoint and let  $B\in {\calL}^2(H_0,H)$ and $C\in {\calL}^2(H,H_0)$ satisfy $[U,V]=iBC$.
Then the product $iBC$ is trace class and skew self-adjoint, and by some straightforward manipulations, one can show that
   \[
   \det \Bigl( (\lambda I-U)(\mu I-V)(\lambda I-U)^{-1} 
      (\mu I-V)^{-1}\Bigr)
      =\det \Bigl( I+iC(\lambda I-U)^{-1}(\mu I-V)^{-1}B\Bigr),\]
where the left-hand side is the determinant of a multiplicative commutator. The scope of this formula is extended via Pincus's principal function, namely the function $P$ introduced in the following lemma.

\begin{lem}\label{Principal} (Pincus) Let $Z=U+iV$ be an almost normal operator. 
\begin{enumerate}[(i)] 
\item There exists a compactly supported $P\in L^1(\Rb^2; dxdy; \Rb )$ that satisfies
  \[ 
    \det \Bigl( (\lambda I-U)(\mu I-V)(\lambda I-U)^{-1} (\mu I-V)^{-1}\Bigr) 
    = \exp\Bigl(  {\frac{1}{2\pi i}}\iint 
       {\frac{P(x,y)}{(x-\lambda )(y-\mu )}} \,dxdy\Bigr).
  \]
   \item In particular suppose that  $B=C^\dagger$ where $C:H\rightarrow \Cb$ has rank one with $[U,V]=iBC$. Then the function $P$ in (i) is supported on $\sigma(U)\times \sigma(V)$, takes values in $[0,1]$ and satisfies  
\[ \det E(\lambda , \mu )
 = \exp\Bigl( {\frac{1}{2\pi i}}\iint_{\sigma (U)\times\sigma (V)} {\frac{P(x,y)}{(\lambda -x)(\mu -y)}} \, dxdy \Bigr).\]
\end{enumerate}
\end{lem}
\begin{proof}  (i) See \cite{CP}.

(ii) See \cite[Theorem 7.1]{Pi}. 
\end{proof}

\begin{prop}\label{cryptotrace} Suppose as in Corollary \ref{crypto}(iii) that $A$ and $A^\dagger$ are almost resolvent commuting, and let $Z=(I-A)(I+A)^{-1}$.
\begin{enumerate}[(i)]
\item Then $Z$ is almost normal and $Z=U+iV$ where $U,V\in \calL (H)$ are self-adjoint with a principal function as in 
Lemma~\ref{Principal}(i) such that
\begin{equation}\label{traceformula} 
   {\trace}\bigl([ f(U,V), h(U,V)]\bigr) 
   ={\frac{1}{2\pi i}}\iint_{\sigma (U)\times\sigma (V)}{\frac{\partial (f,h)}{\partial (x,y)}} P(x,y) \,dxdy
\end{equation}
for all polynomials $f(x,y), h(x,y)\in \Cb [x,y]$.
\item There exists $H_0$ and a continuous time linear system $(-A,B,C)$ with input and output space $H_0$ such that $\calE={\hbox{alg}}\{ I, Z, Z^\dagger\}$ is crypto-integral of dimension one with $2BC=[Z^\dagger, Z]$.
\item In particular, if $[Z^\dagger, Z]$ has rank one, then one can choose $(-A,B,C)$ to have input and output space $\Cb$.\end{enumerate}
\end{prop} 

\begin{proof} (i) Under the canonical quotient map $\calL (H)\rightarrow \calL (H)/\calL^1(H)$, the images of $Z=(I-A)(I+A)^{-1}$ and $Z^\dagger =(I-A^\dagger )( I+A^\dagger)^{-1}$ give commuting elements of the algebra $\calL (H)/\calL^1 (H)$, so $Z$ is almost normal and as such has an essential spectrum in $\Cb$. Now $[Z^\dagger ,Z]=2i[U,V]\in \calL^1(H)$ is self-adjoint, so (i) of Lemma~\ref{Principal} applies.

(ii) We can choose $H_0$ and 
$B\in {\calL}^2(H_0, H)$ and $C\in \calL^2(H, H_0)$ such that $2BC=[Z^\dagger ,Z]$. Observe that ${\hbox{alg}}\{ I, Z, Z^\dagger\}$ is crypto-integral of dimension one, and contains $BC$.

(iii) If $[Z^\dagger ,Z]$ has rank one, then we can choose $B\in {\calL}(\Cb, H)=H$ and $C\in \calL (H,\Cb )=H'$ such that $2BC=[Z^\dagger, Z]$.
\end{proof}

Proposition~\ref{cryptotrace}(i) gives a formula for the cocycle $\varphi_1$ of Theorem \ref{Lyapunovthm}(ii) in a special case, where the principal function is defined on $\Cb$. As discussed in \cite[Theorem 5.10]{CP2}, this is not the typical situation. 
 
\section{Spectral curves}\label{S:Spectral}

Let $(-A,B,C)$ be a linear system as in Lemma~\ref{sectorial}, and consider the potential $u(x)=-4\lfloor A\rfloor_x$. As before, let $L=-{\frac{d^2}{dx^2}} +u$ be the corresponding Schr\"odinger operator. The main result in this section, Theorem~\ref{BurchnallChaundy}, gives conditions on $(-A,B,C)$ such that one can identify the algebra generated by a certain family of differential operators with an algebra of functions on a particular hyperelliptic curve. This is consistent with the notion of a multiplier curve, as in \cite{McK}.

We consider the spectral curve of the corresponding Schr\"odinger operator, and identify conditions under which the spectral curve is a hyperelliptic curve. In this case, the Schr\"odinger's equation is integrable, in the sense that one can integrate $L\psi =\lambda\psi$ by Liouville's operations. To make this precise, we first need to introduce the concept of a Liouville extension of a differential field of functions.\par
\indent 
Let $(\Kb ,\partial )$ be a differential field of complex functions that contains $\Cb$ as constants, and adjoin $h$ to form $\Kb (h)$, where either:\begin{enumerate}[(i)]
  \item $h=\int f$ for some $f\in \Kb$, so $\partial h=f$;
  \item $h=\exp \int f$, so $\partial h=fh$; or
  \item $h$ is algebraic over $\Kb$.
\end{enumerate} 
Then $\partial $ on $\Kb$ extends uniquely to $\partial$ on $\Kb (h)$ so $(\Kb (h) ,\partial )$ is a differential field; see \cite[p.~5]{vdPS}. 

\begin{defn} Let $\Kb_1\subseteq \Kb_2\subseteq\dots\subseteq \Kb_n$ be a sequence of differential fields with differential $\partial$ that contain the subfield $\Cb$ of constants and such that $\Kb_j$ arises from 
 $\Kb_{j-1}$ as $\Kb_j=\Kb_{j-1}(h_j)$ where $h_j$ is as in (i), (ii) or (iii) for $j=2, \dots ,n$. Then $(\Kb_j)_{j=1}^n$  is a Liouville tower, and $\Kb_n$ is said to be a Liouville extension of $\Kb_1$; see \cite[1.42]{vdPS}.  
\end{defn}

\begin{prop}\label{Drach} Suppose that $g_0(x;\zeta )=1/2+\lfloor A(\zeta I-A^2)^{-1}\rfloor_x$ and that $g_0(x, \zeta )^2$ is meromorphic as a function of $\zeta$. Then
\begin{enumerate}[(i)]
  \item the squared diagonal Green's function $G(x,x;\lambda )^2$ is meromorphic in $\lambda$; 
  \item the function $h(x)=G(x,x;\lambda )^2$ satisfies
\[ h^2{\frac{d^3h}{dx^3}}={\frac{3h}{2}}{\frac{dh}{dx}}{\frac{d^2h}{dx^2}}+4(u-\lambda )h^2{\frac{dh}{dx}}-{\frac{3}{4}}\Bigl( {\frac{dh}{dx}}\Bigr)^3+4{\frac{du}{dx}}h^3.\]
  \item Suppose that the potential $u = -4 \lfloor A \rfloor_x$ is meromorphic on some domain and let
     \[ \calU = \Cb \Bigl[\lambda, u, {\frac{du}{dx}}, {\frac{d^2u}{dx^2}}, \dots, h,{\frac{dh}{dx}},{\frac{d^2h}{dx^2}}\Bigr]. \] 
  Then $\calU$ is an integral domain. Further, $(\Kb , d/dx)$, the  field of fractions of $\calU$,  is a differential field, and the general solution of $L\psi =\lambda\psi$ is an element of some Liouville extension of $\Kb$.
\end{enumerate}
\end{prop}  

\begin{proof} (i)
  The proposition follows directly from \cite[Theorem 5.4]{BN}.

(ii) The diagonal Green's function is the product of solutions of $L \psi = \lambda \psi$,
hence $g(x)=G(x,x;\lambda )$ satisfies
\[ {\frac{d^3g}{dx^3}}=2{\frac{du}{dx}}g+ 4(u-\lambda ){\frac{dg}{dx}},\]
and the required differential equation follows by setting $h=g^2$ and considering $h{\frac{dh}{dx}}{\frac{d^2h}{dx^2}}$.

(iii) Since $u$ and $h$ are meromorphic, $\Cb [\lambda, u, {\frac{du}{dx}}, {\frac{d^2u}{dx^2}},\dots, h,{\frac{dh}{dx}},{\frac{d^2h}{dx^2}}]$ is an integral domain, which has a field of fractions $\Kb$. By (ii), $\Kb$ is a differential field, and we can adjoin 
$g$ with $g^2=h$ to form a differential field $\Kb [g]$, and likewise $\Kb [\sqrt{g}]$. Then for some $\mu\in \Cb$, we have 
\[ \mu^2= -2^{-1}g{\frac{d^2g}{dx^2}}+4^{-1}\Bigl({\frac{dg}{dx}}\Bigr)^2+g^2(u-\lambda ),\]
and one can verify that
\begin{equation}\label{generalsolution}
  \psi (x;\lambda )
    =\sqrt {g(x)}\exp \Bigl( \int_0^x {\frac{\mu}{g(y)}}dy \Bigr)
\end{equation}
gives a solution of $L\psi =\lambda \psi.$
\end{proof}

Proposition \ref{Drach} is related to Drach's notion of integrability as in \cite{Br}, but covers a far wider range of potentials than those usually regarded as integrable.
Indeed, we start from a domain $\Cb [\lambda, u,{\frac{du}{dx}}, {\frac{d^2u}{dx^2}}, \dots ,h,{\frac{dh}{dx}},{\frac{d^2h}{dx^2}}]$ which is not necessarily Noetherian
and the general solution (\ref{generalsolution}) is generally not meromorphic as a function of $x$; Gesztesy and Weikard discuss how these ideas are related \cite{GW}. In the next theorem, we address more standard notions of integrability, with classical geometrical interpretations. 

The diagonal Green's function contains much more information than simply the potential. Indeed, it acts as a generating function with coefficients that are differential polynomials in $u$ with universal coefficients which we will identify in the following KdV recursion.
Doubling the coefficients of the term in parentheses in (\ref{diagonalgreen}), we consider the functions $f_0=1$, $f_m=(-1)^m2\lfloor A^{2m-1}\rfloor_x$ for $m=1, 2, \dots$. The differential algebra generated by $u$ and its derivatives with respect to $x$ is $\calR=\Cb [u,{\frac{du}{dx}},{\frac{d^2u}{dx^2}},\dots ]$. Following \cite{GH}, we regard $\lambda,\zeta $ as variables, and define $F_n(\lambda )\in \calR [\lambda ]$ by
   \[ F_n(x;\lambda )=\sum_{j=0}^n f_{n-j}(x)\lambda^j .\]
These $\calR$-valued polynomials satisfy several identities consequent on the stationary KdV hierarchy. We let $Q_{2n+1}(\lambda )\in \calR [\lambda ]$ of degree $2n+1$ be the polynomial
  \begin{equation}\label{Q} 
   Q_{2n+1}(\lambda ) = {\frac{1}{2}} \Bigl( {\frac{\partial^2F_n(x;\lambda )}{\partial x^2}}\Bigr)F_n(x;\lambda )
        -{\frac{1}{4}} \Bigl( {\frac{\partial F_n(x;\lambda )}{\partial x}}\Bigr)^2-\bigl(u(x)-\lambda \bigr)F_n(x;\lambda )^2,
  \end{equation}
and $P_{2n+1}$ be the differential operator 
 of order $2n +1$ in $\partial/\partial x$ with coefficients in $\calR,$
  \begin{equation}\label{P} 
    P_{2n+1} = \sum_{j=1}^n \Bigl( f_{n-j}(x){\frac{\partial}{\partial x}}-{\frac{1}{2}}{\frac {\partial f_{n-j}(x)}{\partial x}}\Bigr) L^j.
  \end{equation}
 We now give the promised version of Burchnall--Chaundy's theorem which gives conditions for a pair of commuting differential operators to satisfy a polynomial equation and thereby determine an algebraic curve.
 
\begin{thm}\label{BurchnallChaundy} 
For $(-A,B,C)$ as in Lemma \ref{sectorial} and let $(f_j)$ be as in (\ref{P}) and suppose that there exists an integer $\ell \ge 1$ such that $\frac{\partial f_{\ell +1}}{\partial x} = 0$ but $\frac{\partial f_{\ell }} {\partial x} \ne 0$. Then 
\begin{enumerate}[(i)]
  \item the differential operators satisfy $Q_{2\ell +1}(L)=-P_{2\ell +1}^2$;
 \item The function $f(\lambda,\zeta) = \zeta^2-Q_{2\ell +1}(\lambda )$ is independent of $x$ and so $f(\lambda,\zeta) = 0$ determines a hyperelliptic curve $\Xb$ of genus $\ell$;
\item the commutative algebra of ordinary differential operators that is generated by $L$ and $P_{2\ell+1}$ is isomorphic to the algebra of regular functions on $\Xb$.
\end{enumerate}
\end{thm}

\begin{proof} (i) In \cite[Proposition 5.2]{BN}, it is shown that the sequence $(f_m)_{m=0}^\infty$ lies in $\calR$ and satisfies the recursion relation of the stationary KdV hierarchy. By \cite[Remark 1.5]{GH}, there exists a solution of the stationary KdV hierarchy that terminates in the sense there is a solution $(f_m)_{m=0}^\infty $ with only finitely many nonzero $f_m$, $m=0,1, \dots ,\ell$; hence $L$ commutes with $P_{2\ell +1}$. Note that 
$\calR=\Cb [u,{\frac{\partial u}{\partial x}},{\frac{\partial^2u}{\partial x^2}},\dots , {\frac{\partial^{2\ell} u}{\partial x^{2\ell }}}]$, giving a Noetherian differential ring.
On substituting $L$ for $\lambda$ in $Q_{2\ell +1}(\lambda )$, one obtains a differential operator $Q_{2\ell +1}(L)$ of order $4\ell +2$. The result now follows as in the proof of \cite[Theorem 1.3]{GH}.

(ii) Let $f(\lambda, \zeta)=\zeta^2-Q_{2\ell +1}(\lambda)$, which independent of $x$ by the hypothesis on $f_{\ell+1}$ and \cite[p.~7]{GH}. Also, $f(\lambda, \zeta )$ is irreducible in $\Cb [\lambda, \zeta ]$.

(iii) Let $\calA_L$ be the set of differential operators that commute with $L$, so $\calA_L$ is a commutative algebra by a result of Schur. We have seen that 
$P_{2\ell +1}\in \calA_L$. This implies that the greatest common divisor of the orders of the operators in $\calA_L$ is one, and so $\calA_L$ has rank one \cite[p.~182]{W}. Let ${\hbox{Specm}}(\calA_L)$ be the space of maximal ideals of $\calA_L$ which is a subset of the space ${\hbox{Spec}}(\calA_L)$ of prime ideals. Burchnall and Chaundy \cite{BC} showed that ${\hbox{Specm}}(\calA_L )$ is an irreducible algebraic curve that can be completed by adding one smooth point at infinity; see \cite[p.~182]{W} and \cite[p.~132]{Mum2} for a modern presentation of this result. 

Then $\Cb [L, P_{2\ell +1}]$ is isomorphic to $\Cb [\lambda, \zeta]/(f(\lambda, \zeta ))$, namely the algebra generated by the coordinate functions on the hyperelliptic curve $\Xb$; see  \cite[Remark 2.8]{PRZ}.  One can show that $\Cb [L, P_{2\ell +1}]$ is a complex subalgebra of $\calA_L$ such that the  dimension of the quotient vector space $\calA_L/\Cb [L, P_{2\ell +1}]$ is finite; see \cite[section 2]{Mul}.
We interpret $\Xb$ as sheets covering $\Cb$. There is an inclusion $\Cb [L]\rightarrow \calA_L$, hence a surjective map of prime ideals ${\hbox{Spec}}(\calA_L)\rightarrow {\hbox{Spec}}(\Cb [L])$ given by $P\mapsto P\cap \Cb [L]$; see \cite[p.~223]{Sha}. Thus a spectral point $\lambda\in \Cb$ gives the prime ideal generated by $L - \lambda I$ in $\Cb [L]$ which is covered by a prime ideal of $\calA_L$.
\end{proof}

\begin{rem} In \cite[Theorem 3.1]{Br}, Brezhnev shows that if the system $(-A,B,C)$ satisfies Theorem~\ref{BurchnallChaundy} then the corresponding $L$ is integrable, in the sense that $L\psi=\lambda \psi$ can be solved explicitly by Liouville integration.
\end{rem}

The curve $\Xb$ is hyperelliptic, so there exists a meromorphic $\wp:\Xb \rightarrow\Cb_\infty$ that is typically $2:1$, with branch points at $\{ p\in \Xb: \wp'(p)=0\}$; see \cite[3.155]{Mum}. We introduce the real points on $\Xb$ by $\Xb_r=\{ (\lambda , \zeta )\in \Xb: \lambda ,\zeta \in \Rb\}$, which typically gives a disconnected subset which projects to $\{ \lambda \in \Rb: Q_{2\ell+1}(\lambda )>0\}$ via $(\lambda, \zeta )\mapsto \lambda$. 

Suppose that $u$ is real, and consider how $\Xb_r$ and $\sigma (L)$ are related. In \cite[Proposition 3.2]{BN} we provide the Baker--Akhiezer function $\psi_\lambda$ such that $L\psi_\lambda =\lambda\psi_\lambda$, so $\{ \lambda\in \Rb : \psi_\lambda \in L^\infty ((0, \infty ); \Cb )\}$ is the set of approximate eigenvalues, contained in the $L^2$ spectrum of $L$.
Then ${\hbox{int}}\{ \lambda\in \Rb: \psi_\lambda\in L^\infty ((0, \infty ): \Cb )\}$ is contained in the essential spectrum $\sigma_{ess}(L)$. The function $-G(x,x;\lambda )$ is holomorphic on $\Cb\setminus \sigma (L)$. For real potentials that also satisfy Theorem \ref{BurchnallChaundy}, we have $\wp :\Xb\rightarrow \Cb_\infty $ restricting to $\wp :\Xb_r\rightarrow \sigma_{ess}(L)$.  

Let $\Xb_0$ be an open subset of $\Xb$ such that the boundary $\partial \Xb_0$ of $\Xb_0$ is a finite union of simple analytic curves. Let $w_0\in\Xb_0$, and let $\nu_0$ be the harmonic measure on $\partial\Xb_0$ with respect to $w_0$. Then we introduce the complex unital algebra $\calR_0$ of meromorphic functions on $\Xb$ that are holomorphic on $\Xb_0\cup\partial\Xb_0$. We let $H=L^2(\nu_0;\Cb )$ and regard $\calR_0$ as an algebra of multiplication operators on $H$; then we let $H^2$ be the closure of $\{ f: f\in \calR_0\}$ in $H$, so $H^2$ is a Hilbert module over $\calR_0.$ Let $R:L^2\rightarrow H^2$ be the orthogonal projection; then we write $T_f\in\calL (H^2)$ for the operator $T_f:h\mapsto fh$ for $h\in H^2$ and $T_f^\dagger\in \calL (H^2)$ for the adjoint. We also introduce $\Gamma_f:L^2\ominus H^2\rightarrow H^2$ as $\Gamma_f=RM_f(I-R)$ where $M_f:L^2\rightarrow L^2$ is the multiplication operator $M_f:h\rightarrow fh$. For $f\in \calR_0$, this $M_f$ has the block form
\begin{equation} \label{Toeplitzblock2}M_f=\begin{bmatrix} T_{f}&\Gamma_{f}\\ 0& \tilde T_{f}\end{bmatrix}\qquad \begin{matrix} H^2\\ L^2\ominus H^2\end{matrix}\end{equation}
which may be contrasted with (\ref{Toeplitzblock}).

\begin{cor}\label{hyperprincipal}\begin{enumerate}[(i)]

\item Let $\calE_0$ be the $\ast$ algebra of operators on $H^2$ that is generated by $\{ T_f, T_f^\dagger: f\in\calR_0\}$ and $\calL^1(H^2)$; then there is an exact sequence of complex algebras and homomorphisms 
\begin{equation}0\longrightarrow \calL^1(H^2)\longrightarrow\calE_0\longrightarrow \calC\longrightarrow 0\end{equation}
where $\calC$ is a commutative and unital algebra.
\item For $f\in\calR_0$, let $T_f=X+iY$ where $X,Y\in \calL (H^2)$ are self-adjoint. Then there exists a nonnegative, compactly supported and integrable function $P_f:\Rb^2\rightarrow \Rb$ such that
\begin{equation}\trace \bigl([ p(X,Y), q(X,Y)]\bigr)={\frac{1}{2\pi i}}\iint_{\Rb^2}{\frac{\partial (p,q)}{\partial (x,y)}} P_f(x,y)\,dxdy\end{equation}
for all polynomials $p(x,y), q(x,y)\in \Cb [x,y]$. 
\item $\Gamma_f\in \calL^2(L^2\ominus H^2, H^2)$ for all $f\in \calR_0$.
\end{enumerate}
\end{cor}
\begin{proof} (i) By (\ref{Toeplitzblock2}), the map $f\mapsto T_f$ is a unital algebra homomorphism $\calR_0\rightarrow \calL (H^2)$. By \cite[Theorem 5]{DY}, we have $[T_{f_0}^\dagger, T_{f_1}]\in \calL^1(H^2)$ for all $f_0,f_1\in\calR_0$. Then $[\calE_0,\calE_0]\subset \calL^1(H^2)$, so the quotient algebra $\calE_0/\calL^1(H^2)$ is a commutative and unital complex algebra.\par
(ii) In particular, for $f\in\calR_0$, the operator $T_f$ is almost normal with $T_f=X+iY$ with self-adjoint $X,Y\in \calL (H^2)$ such that $2i[X,Y]=[T_f^\dagger ,T_f]$. Thus from Lemma \ref{Principal}, we obtain a trace formula on the $\ast$ algebra generated by $T_f$.\par 
(iii) We have $M_fM_f^\dagger -M_f^\dagger M_f=0$, so from (\ref{Toeplitzblock2}), we have $\Gamma_f\Gamma_f^\dagger =T_f^\dagger T_f-T_fT_f^\dagger$ is trace class; hence $\Gamma_f$ is Hilbert-Schmidt. To evaluate the squared norm, we apply the trace formula with $p(x,y)=x-iy$ and $q(x,y)=x+iy$, 
\begin{align}\Vert \Gamma_f\Vert^2_{\calL^2}&=\trace \bigl( [T_f^\dagger ,T_f]\bigr)\nonumber\\
&={\frac{1}{\pi}}\iint_{\Rb^2} P_f(x,y) \, dxdy.\end{align}
Compare Proposition \ref{cryptotrace}.
\end {proof}

Computing the principal function $P_f$ is a challenge. Under the following circumstances, one can be more explicit about the elements of $\calR_0$.

\begin{ex} As in (\ref{gaps}), suppose that $\sigma_0 =[c_1, d_1]\cup\dots \cup [c_n, \infty )$ is a finite union of disjoint real intervals. Then there is a linear fractional transformation taking $\sigma_0\cup \{ \infty \}$ bijectively to $\sigma =\cup_{j=1}^n[a_j,b_j]$, where $[a_j,b_j]$ are disjoint real intervals for $j=1, \dots, n$, which are contained in a bounded open disc $\Db (0,r)$ centered at the origin. For the domain $\Db (0,r)\setminus \sigma$, one can construct the hydrodynamic Green's function by procedures described in detail in \cite{CM}, and the harmonic measure is given by its boundary derivative.  Let $\nu_\sigma$ be the harmonic measure of $\sigma$ at infinity in $\Cb\cup \{ \infty\}$, and suppose further that $\nu_\sigma ([a_j,b_j])$ is rational for all $j=1, \dots, n$. Then by \cite[Lemma 2.2]{To} there exists a holomorphic function $h: \Cb\setminus \sigma\rightarrow\Cb$ such that $\vert h(x)\vert =1$ for all $x\in\sigma$ and $h(x)\in (-\infty ,-1)\cup (1, \infty )$ for all $x\in \Rb\setminus \sigma$; indeed, we can choose a positive integer $N$ such that $N\nu_\Omega ([a_j,b_j])$ is an integer for $j=1, \dots, n$, and define 
   \begin{equation}
   h(z)=\exp N\Bigl( \int_\sigma\log (z-t) \nu_\sigma (dt)-\log {\hbox{cap}}(\sigma )\Bigr)
   \end{equation}
where ${\hbox{cap}} (\sigma )$ is the logarithmic capacity of $\sigma$. Then $g(z)=2^{-1}(h(z)+h(z)^{-1})$ is a polynomial such that  $h(z)^2-2g(z)h(z)+1=0$, and
$\sigma =\{ x\in \Rb : -1\leq g(x)\leq 1\}$; then $\{ (x, \pm \Im h(x)):x\in [a_j,b_j]\}$ describes an oval. Note that $\Xb=\{ (z,w): w^2-2g(z)w+1=0\}$ gives an algebraic curve which we identify with a compact Riemann surface and an algebraic function as in \cite[page 250]{FK}. As 
 in \cite[IV.11]{FK} and \cite[3.155]{Mum}, we suppose that $\wp$ is a meromorphic function on $\Xb$ such that $(\wp ,\wp')$ gives a point on the curve, and $\wp$ is holomorphic on the closure of $\Xb_0$, where $\Xb_0 =\{ \zeta \in \Xb: \wp (\zeta )\in \Db (0,r)\setminus \sigma \}$. Then we  choose $w_0\in \Xb_0$ and form the Hardy space $H^2$ for the harmonic measure $\nu_0$ of $\partial \Xb_0$ with respect to $w_0$. Then $Z=T_\wp$ and $W=T_{\wp'}$ are commuting linear operators.
\end{ex} 
\textbf{Acknowledgement} The first author thanks Jonathan Arazy for an illuminating discussion on this topic several years ago.


\end{document}